\documentclass[reqno]{amsart}
\usepackage{amssymb,amsmath,hyperref}
\usepackage{amsrefs, float}
\usepackage[foot]{amsaddr}
\usepackage{bbold,stackrel, multicol}
 
\newtheorem{thm}{Theorem}

\newtheorem{cor}[thm]{Corollary}
\newtheorem{defi}[thm]{Definition}
\newtheorem{claim}[thm]{Claim}
\newtheorem{rem}[thm]{Remark}
\newtheorem{nota}[thm]{Notation}

\newtheorem{princ}[thm]{Principle}

\newtheorem{ack}[thm]{Acknowledgement}

\newtheorem*{tempo*}{Template}
\newtheorem{theorem}[thm]{Theorem}

\newtheorem{lemma}[thm]{Lemma}
\newtheorem{definition}[thm]{Definition}
\newtheorem{corollary}[thm]{Corollary}
\newtheorem{remark}[thm]{Remark}

\newcommand\be{\begin{equation}}
\newcommand\ee{\end{equation}} 

\usepackage{amsmath,amsfonts} 

\usepackage[applemac]{inputenc}

\usepackage{multicol}

\def\bdefi{\begin{defi}\rm}
\def\edefi{\end{defi}}
\def\bnota{\begin{nota}\rm}
\def\enota{\end{nota}}

\def\ZFC{\textup{\textsf{ZFC}}}

 \def\r{\mathbb{r}}

\def\SUP{\textup{\textsf{sup}}}

\def\({\textup{(}}
\def\){\textup{)}}

\def\bye{\end{document}}
\def\P{\textup{\textsf{P}}}
\def\N{{\mathbb  N}}
\def\Q{{\mathbb  Q}}
\def\R{{\mathbb  R}}

\def\SS{\textup{\textsf{S}}}

\def\J{\mathcal{J}}
\def\D{{\mathcal  D}}

\def\di{\rightarrow}

\def\asa{\leftrightarrow}

\def\ACA{\textup{\textsf{ACA}}}

\def\fix{\textup{\textsf{fix}}}
\def\SUP{\textup{\textsf{sup}}}
\def\ii{\textup{\textsf{i}}}

\def\KL{\textup{\textsf{KL}}}

\def\w{\textup{\textsf{w}}}
\def\BW{\textup{\textsf{BW}}}
\def\BV{\textup{\textsf{BV}}}

\def\fin{\textup{\textsf{fin}}}

\def\eps{\varepsilon}

\newcommand{\F}{{\bf F}}

\newcommand{\rec}{\mathsf{Rec}}

\newcommand{\ra}{{\bf r}}
\newcommand{\Ty}{{\bf Ty}}

\newcommand{\Pa}{{\bf P}}
\newcommand{\suc}{\mathsf {suc}}
\newcommand{\pd}{\mathsf {pred}}
\newcommand{\case}{\mathsf {case}}

\usepackage{graphicx}
\usepackage{tikz}
\usetikzlibrary{matrix, shapes.misc}
\usepackage{comment,tikz-cd}

\numberwithin{equation}{section}
\numberwithin{thm}{section}

\begin{document}
\title[]{On the computational properties of basic mathematical notions}
\author{Dag Normann}
\address{Department of Mathematics, The University 
of Oslo, P.O. Box 1053, Blindern N-0316 Oslo, Norway}
\email{dnormann@math.uio.no}
\author{Sam Sanders}
\address{Department of Philosophy II, RUB Bochum, Germany}
\email{sasander@me.com}
\begin{abstract}
We investigate the computational properties of basic mathematical notions pertaining to $\R\di \R$-functions and subsets of $\R$, like \emph{finiteness}, \emph{countability}, \emph{\(absolute\) continuity}, \emph{bounded variation}, \emph{suprema}, and \emph{regularity}.  We work in higher-order computability theory based on Kleene's S1-S9 schemes.
We show that the aforementioned italicised properties give rise to two huge and robust classes of computationally equivalent operations, the latter based on well-known theorems from the mainstream mathematics literature.  
As part of this endeavour, we develop an equivalent $\lambda$-calculus formulation of S1-S9 that accommodates \emph{partial} objects.  We show that the latter are essential to our enterprise via the study of countably based and partial functionals of type $3$.  We also exhibit a connection to \emph{infinite time} Turing machines.
\end{abstract}


\maketitle
\thispagestyle{empty}

\section*{CORRIGENDUM}
The lambda calculus introduced in Section \ref{maink} unfortunately suffers from a technical error.  The latter was communicated to us in a private communication by John Longley.  
In particular, the restriction in the fifth bullet point of Definition \ref{intdef} is too strong. As a consequence, Theorem~\ref{prize} is not correct.  A corrected version may be found in \cite{dagsamXV}*{Section 5}.
The computability theoretic results in this paper remain unaffected.  

\section{Introduction}\label{intro}
Given a finite set, perhaps the most basic questions are \emph{how many} elements it has, and \emph{which ones}?  
We study this question in Kleene's \emph{higher-order computability theory}, based on his computation schemes S1-S9 (see \cites{kleeneS1S9, longmann}).  
In particular, a central object of study is the higher-order functional $\Omega$ which on input a finite set of real numbers, list the elements as a finite sequence. 

\smallskip

Perhaps surprisingly, the `finiteness' functional $\Omega$ give rise to a \emph{huge and robust} class of computationally equivalent operations, called the $\Omega$-cluster, as explored in Section \ref{maink2}.  
For instance, many basic operations on functions of \emph{bounded variation} (often abbreviated to `$BV$' in the below) are part of the $\Omega$-cluster, including those stemming from the well-known \emph{Jordan decomposition theorem}  (see Theorem \ref{drd}).  
We have studied the computational properties of the latter theorem in \cite{dagsamXII} and this paper greatly extends the results in the latter. 
In addition, we identify a second cluster of computationally equivalent objects, called the $\Omega_{1}$-cluster, based on $\Omega_{1}$, the restriction of $\Omega$ to singletons.  
We also show that both clusters include basic operations on \emph{regulated} and \emph{Sobolev space} functions, respectively a well-known super- and sub-class of the class of $BV$-functions. 

\smallskip

As will become clear, our objects of study are fundamentally \emph{partial} in nature, and we formulate an elegant and equivalent $\lambda$-calculus formulation of S1-S9 to accommodate partial objects (Section~\ref{maink}).    
The advantages of this approach are three-fold: proofs are more transparent in our $\lambda$-calculus approach, all (previously hand-waved) technical details can be settled easily, and we can show that $\Omega_{1}$ and $\Omega$ are not computationally equivalent to any \emph{total} functional.

\smallskip

As to the broader context of our endeavour, Jordan introduces the notion of $BV$-function around 1881 in \cite{jordel}, while Lakatos claims in \cite{laktose} that Dirichlet's original 1829 proof from \cite{didi3} already contains this notion.  
Moreover, we also study the class of \emph{rectifiable functions}, i.e.\ the largest class for which the notion of `arc length of the graph' makes sense, which coincides with the class of $BV$-functions, as shown in \cite{voordedorst}*{Ch.\ 1}.  
However, the modern notion of arc length goes back to 1833-1866, as discussed in Section \ref{flukkkkq}.
Thus, our study of the $\Omega$-cluster has a clear historical angle.      
Moreover, the following quote motivates that functions of bounded variation  constitute `ordinary' or `mainstream' mathematics.
\begin{quote}
The space $BV$, consisting of functions with bounded variation, is of particular interest for applications to data compression and statistical estimation. It is often chosen as a model for piecewise smooth signals such as geometric images. (\cite{cohenn}*{p.\ 236})
\end{quote}
Similar claims can be made for Sobolev spaces and regulated functions, prominent/essential as they are in the study of PDEs or Riemann integration.  

\smallskip

Finally, all technical notions are introduced in Section \ref{prek} while our main results are in Section \ref{maink} and \ref{maink2}.  
The informed reader will know that the computational properties of $BV$-functions have been studied via second-order representations, as also mentioned in Section~\ref{defki}.

\section{Preliminaries}\label{prek}
We introduce some required background, including Kleene's computational framework (Section \ref{kle9}) and some higher-order definitions (Section \ref{klonkio}), like the notion of $BV$-function and related concepts.  
Before all that, we remind the reader that higher-order computability theory is generally formulated \emph{in the language of type theory}.  
The well-known notion of `order' of an object corresponds to that of `type rank', but with difference $1$ (see Definition \ref{def1}). 
We do not always distinguish between the `type' of an object and its `type rank', to avoid cumbersome notations.  

\subsection{Kleene's higher-order computability theory}\label{kle9}
We first make our notion of `computability' precise as follows.  
\begin{enumerate}
\item[(I)] We adopt $\ZFC$, i.e.\ Zermelo-Fraenkel set theory with the Axiom of Choice, as the official metatheory for all results, unless explicitly stated otherwise.
\item[(II)] We adopt Kleene's notion of \emph{higher-order computation} as given by his nine schemes S1-S9 (see \cite{longmann}*{Ch.\ 5} or \cite{kleeneS1S9}) as our official notion of `computable'.
\end{enumerate}
We mention that S1-S8 are rather basic and merely introduce a kind of higher-order primitive recursion with higher-order parameters. 
The real power comes from S9, which essentially hard-codes the \emph{recursion theorem} for S1-S9-computability in an ad hoc way.  
By contrast, the recursion theorem for Turing machines is derived from first principles in \cite{zweer}.

\smallskip

On a historical note, it is part of the folklore of the subject of computability theory that many have tried (and failed) to formulate models of computation for objects of all finite types in which one derives the recursion theorem in a natural way.  For this reason, Kleene ultimately introduced S1-S9, which 
were initially criticised for their ad hoc nature, but eventually received general acceptance nonetheless.  

\smallskip

We refer to \cite{longmann} for a (more) thorough overview of higher-order computability theory.
We do mention the distinction between `normal' and `non-normal' functionals  based on the following definition from \cite{longmann}*{\S5.4}. 
\bdefi\label{norma}
For $n\geq 2$, a functional of type $n$ is called \emph{normal} if it computes Kleene's $\exists^{n}$ following S1-S9, and \emph{non-normal} otherwise.  
\edefi
\noindent
We only make use of $\exists^{n}$ for $n=2,3$, as defined in Section \ref{prelim1}.

\smallskip

It is a historical fact that higher-order computability theory, based on Kleene's S1-S9 schemes, has focused primarily on the world of \emph{normal} functionals; this opinion can be found \cite{longmann}*{\S5.4}.  
Nonetheless, we have previously studied the computational properties of new \emph{non-normal} functionals, namely those that compute the objects claimed to exist by:
\begin{itemize}
\item covering theorems due to Heine-Borel, Vitali, and Lindel\"of (\cites{dagsam, dagsamII, dagsamVI}),
\item the Baire category theorem (\cite{dagsamVII}),
\item local-global principles like \emph{Pincherle's theorem} (\cite{dagsamV}),
\item the uncountability of $\R$ and the Bolzano-Weierstrass theorem for countable sets in Cantor space (\cites{dagsamX, dagsamXI}),
\item weak fragments of the Axiom of (countable) Choice (\cite{dagsamIX}),
\item the Jordan decomposition theorem and other results on $BV$-functions (\cite{dagsamXII}).
\end{itemize}
In this paper, we greatly extend the study mentioned in the final item.  
Next, we introduce some required higher-order notions in Section~\ref{klonkio}.

\subsection{Some higher-order notions}\label{klonkio}
We introduce some comprehension functionals (Section \ref{prelim1}) and definitions (Section \ref{defki}) that are essential to the below.
We first introduce the usual notations for real numbers and finite sequences. 

\subsubsection{Notations and the like}
We introduce the usual notations for common mathematical notions, like real numbers, as also introduced in \cite{kohlenbach2}.  
\begin{defi}[Real numbers and related notions]\label{keepintireal}\rm~
\begin{enumerate}
 \renewcommand{\theenumi}{\alph{enumi}}
\item Natural numbers correspond to type zero objects, and we use `$n^{0}$' and `$n\in \N$' interchangeably.  Rational numbers are defined as signed quotients of natural numbers, and `$q\in \Q$' and `$<_{\Q}$' have their usual meaning.    
\item Real numbers are coded by fast-converging Cauchy sequences $q_{(\cdot)}:\N\di \Q$, i.e.\  such that $(\forall n^{0}, i^{0})(|q_{n}-q_{n+i}|<_{\Q} \frac{1}{2^{n}})$.  
We use Kohlenbach's `hat function' from \cite{kohlenbach2}*{p.\ 289} to guarantee that every $q^{1}$ defines a real number.  
\item We write `$x\in \R$' to express that $x^{1}:=(q^{1}_{(\cdot)})$ represents a real as in the previous item and write $[x](k):=q_{k}$ for the $k$-th approximation of $x$.    
\item Two reals $x, y$ represented by $q_{(\cdot)}$ and $r_{(\cdot)}$ are \emph{equal}, denoted $x=_{\R}y$, if $(\forall n^{0})(|q_{n}-r_{n}|\leq {2^{-n+1}})$. Inequality `$<_{\R}$' is defined similarly.  
We sometimes omit the subscript `$\R$' if it is clear from context.           
\item Functions $F:\R\di \R$ are represented by $\Phi^{1\di 1}$ mapping equal reals to equal reals, i.e.\ extensionality as in $(\forall x , y\in \R)(x=_{\R}y\di \Phi(x)=_{\R}\Phi(y))$.\label{EXTEN}
\item Equality at higher types is defined in terms of equality between naturals `$=_{0}$' as follows: for any objects $x^{\tau}, y^{\tau}$, we have
\be\label{aparth}
[x=_{\tau}y] \equiv (\forall z_{1}^{\tau_{1}}\dots z_{k}^{\tau_{k}})[xz_{1}\dots z_{k}=_{0}yz_{1}\dots z_{k}],
\ee
if the type $\tau$ is composed as $\tau\equiv(\tau_{1}\di \dots\di \tau_{k}\di 0)$.  
\item The relation `$x\leq_{\tau}y$' is defined as in \eqref{aparth} but with `$\leq_{0}$' instead of `$=_{0}$'.  Binary sequences are denoted `$f^{1}, g^{1}\leq_{1}1$', but also `$f,g\in C$' or `$f, g\in 2^{\N}$'.  Elements of Baire space are given by $f^{1}, g^{1}$, but also denoted `$f, g\in \N^{\N}$'.
\item For a binary sequence $f^{1}$, the associated real in $[0,1]$ is $\r(f):=\sum_{n=0}^{\infty}\frac{f(n)}{2^{n+1}}$.\label{detrippe}
\end{enumerate}
\end{defi}
For completeness, we list the following notational convention for finite sequences.  
\begin{nota}[Finite sequences]\label{skim}\rm
The type for `finite sequences of objects of type $\rho$' is denoted $\rho^{*}$, which we shall only use for $\rho=0,1$.  
Since the usual coding of pairs of numbers is rather elementary, we shall not always distinguish between $0$ and $0^{*}$. 
Similarly, we assume a fixed coding for finite sequences of type $1$ and shall make use of the type `$1^{*}$'.  
In general, we do not always distinguish between `$s^{\rho}$' and `$\langle s^{\rho}\rangle$', where the former is `the object $s$ of type $\rho$', and the latter is `the sequence of type $\rho^{*}$ with only element $s^{\rho}$'.  The empty sequence for the type $\rho^{*}$ is denoted by `$\langle \rangle_{\rho}$', usually with the typing omitted.  

\smallskip

Furthermore, we denote by `$|s|=n$' the length of the finite sequence $s^{\rho^{*}}=\langle s_{0}^{\rho},s_{1}^{\rho},\dots,s_{n-1}^{\rho}\rangle$, where $|\langle\rangle|=0$, i.e.\ the empty sequence has length zero.
For sequences $s^{\rho^{*}}, t^{\rho^{*}}$, we denote by `$s*t$' the concatenation of $s$ and $t$, i.e.\ $(s*t)(i)=s(i)$ for $i<|s|$ and $(s*t)(j)=t(|s|-j)$ for $|s|\leq j< |s|+|t|$. For a sequence $s^{\rho^{*}}$, we define $\overline{s}N:=\langle s(0), s(1), \dots,  s(N-1)\rangle $ for $N^{0}<|s|$.  
For a sequence $\alpha^{0\di \rho}$, we also write $\overline{\alpha}N=\langle \alpha(0), \alpha(1),\dots, \alpha(N-1)\rangle$ for \emph{any} $N^{0}$.  By way of shorthand, 
$(\forall q^{\rho}\in Q^{\rho^{*}})A(q)$ abbreviates $(\forall i^{0}<|Q|)A(Q(i))$, which is (equivalent to) quantifier-free if $A$ is.   
\end{nota}

\subsubsection{Higher-order comprehension functionals}\label{prelim1}
We introduce a number of well-known `comprehension functionals' from the literature.  
We are dealing with \emph{conventional} comprehension, i.e.\ only parameters over $\N$ and $\N^{\N}$ are allowed in formula classes like $\Pi_{k}^{1}$ and $\Sigma_{k}^{1}$.  

\smallskip

First of all, the third-order functional $\varphi$ as in $(\exists^{2})$ is often called `Kleene's quantifier $\exists^{2}$' and we use the same naming convention for other functionals.
\be\label{muk}\tag{$\exists^{2}$}
(\exists \varphi^{2}\leq_{2}1)(\forall f^{1})\big[(\exists n^{0})(f(n)=0) \asa \varphi(f)=0    \big]
\ee
Intuitively speaking, Kleene's $\exists^{2}$ can decide the truth of any $\Sigma_{1}^{0}$-formula in its (Kleene) normal form.  
Related to $(\exists^{2})$, the third-order functional in $(\mu^{2})$ is also called \emph{Feferman's $\mu$} (\cite{avi2}), defined as follows:
\begin{align}\label{mu}
(\exists \mu^{2})(\forall f^{1})\big[ (\exists n)(f(n)=0) \di &[f(\mu(f))=0\wedge (\forall i<\mu(f))(f(i)\ne 0) ]\notag\\
& \wedge [ (\forall n)(f(n)\ne0)\di   \mu(f)=0]    \big]. \tag{$\mu^{2}$}
\end{align}
We have $(\exists^{2})\asa (\mu^{2})$ over a weak system by \cite{kooltje}*{Prop.\ 3.4 and Cor.~3.5}) while $\mu^{2}$ is readily computed from Kleene's $\exists^{2}$.
The operator $\mu^{2}$, with that symbol, plays a central role in Hilbert-Bernays' \emph{Grundlagen} (\cites{hillebilly, hillebilly2}).  Now, the same symbol `$\mu$' is also often used for fixed point operators, a convention we also adopt (see Section~\ref{kleedef}).  
This dual use of `$\mu$' shall not lead to confusion, especially since Feferman's $\mu$ is essentially\footnote{Define $\mu_{\fix}(f):= \mu(\lambda n. (f(n)-n))$ and note that $f(\mu_{\fix}(f))=\mu_{\fix}(f)$ in case $f\in \N^{\N}$ has a fixed point.  The leastness follows by the definition of Feferman's $\mu^{2}$} a (least) fixed point operator.  

\smallskip

Secondly, $\SS^{2}$ as in $(\SS^{2})$ is called \emph{the Suslin functional} (\cites{kohlenbach2, avi2}):
\be\tag{$\SS^{2}$}
(\exists\SS^{2}\leq_{2}1)(\forall f^{1})\big[  (\exists g^{1})(\forall n^{0})(f(\overline{g}n)=0)\asa \SS(f)=0  \big].
\ee
Intuitively, the Suslin functional $\SS^{2}$ can decide the truth of any $\Sigma_{1}^{1}$-formula in its normal form.  
We similarly define the functional $\SS_{k}^{2}$ which decides the truth or falsity of $\Sigma_{k}^{1}$-formulas (again in normal form).  
We note that the operators $\nu_{n}$ from \cite{boekskeopendoen}*{p.\ 129} are essentially $\SS_{n}^{2}$ strengthened to return a witness to the $\Sigma_{n}^{1}$-formula at hand.  
As suggested by its name, $\nu_{k}$ is the restriction of Hilbert-Bernays' $\nu$ from \cite{hillebilly2}*{p.\ 495} to $\Sigma_{k}^{1}$-formulas. 
We sometimes use $\SS_{0}^{2}$ and $\SS_{1}^{2}$ to denote $\exists^{2}$ and $\SS^{2}$.  

\smallskip

\noindent
Thirdly, second-order arithmetic is readily derived from the following:
\be\tag{$\exists^{3}$}
(\exists E^{3}\leq_{3}1)(\forall Y^{2})\big[  (\exists f^{1})(Y(f)=0)\asa E(Y)=0  \big].
\ee
The functional from $(\exists^{3})$ is also called `Kleene's quantifier $\exists^{3}$'.  
Hilbert-Bernays' $\nu$ from \cite{hillebilly2}*{p.\ 495} trivially computes $\exists^{3}$.

\smallskip

Finally, the functionals  $\SS_{k}^{2}$ are defined using the usual formula class $\Pi_{k}^{1}$, i.e.\ only allowing first- and second-order parameters.  
We have dubbed this the \emph{conventional approach} and the associated functionals are captured by the umbrella term \emph{conventional comprehension}. 
Comprehension involving third-order parameters has previously (only) been studied in \cites{littlefef, kohlenbach4}, to the best of our knowledge.

\subsubsection{Some higher-order definitions}\label{defki}
We introduce some standard definitions required for introducing the many inhabitants of the $\Omega$ and $\Omega_{1}$-clusters. 

\smallskip

First of all, a fruitful and faithful approach is the representation of sets by characteristic functions (see e.g.\ \cites{samnetspilot, samcie19,samwollic19,dagsamVI,dagsamVII, kruisje, dagsamIX}), well-known from e.g.\ measure and probability theory and going back to Dirichlet (1830; \cite{didi1}).   We shall use this approach, assuming $\exists^{2}$ to make sure e.g.\ countable unions make sense.  

\smallskip

Secondly, we now turn to \emph{countable sets}.
Of course, this notion can be formalised in various ways, as follows.  
\bdefi[Enumerable set]\label{eni}
A set $A\subset \R$ is \emph{enumerable} if there is a sequence $(x_{n})_{n\in \N}$ such that $(\forall x\in \R)(x\in A\di (\exists n\in \N)(x=x_{n}))$.  
\edefi
We note that Definition \ref{eni} reflects the notion of `countable set' from second-order reverse mathematics (\cite{simpson2}*{V.4.2}).  
\bnota[Enumeration]\label{defornota}
Given Feferman's $\mu^{2}$, we can remove all elements from a sequence $(x_{n})_{n\in \N}$ that are not in a given set $A$, i.e.\ obtain an equivalence in Definition \ref{eni} in case $A$ is enumerable;
the resulting sequence is called an \emph{enumeration} of $A$.  As a technicality, we shall refer to the `null sequence'  $\langle \rangle_{1}*\langle \rangle_{1}*\dots$ as an enumeration of the empty set $\emptyset\subset \R$.
\enota
%
Our definition of `countable set' is now as follows.  
\bdefi[Countable set]\label{standard}~
Any $A\subset \R$ is \emph{countable} if there is $Y:\R\di \N$ with 
\be\label{polki}
(\forall x, y\in A)(Y(x)=_{0}Y(y)\di x=y).
\ee 
The functional $Y$ as in \eqref{polki} is called \emph{injective} on $A$ or \emph{an injection} on $A$.
If $Y:\R\di \N$ is also \emph{surjective}, i.e.\ $(\forall n\in \N)(\exists x\in A)(Y(x)=n)$, we call $A$ \emph{strongly countable}.
The functional $Y$ is then called \emph{bijective} on $A$ or \emph{a bijection} on $A$.
\edefi
The first part of Definition \ref{standard} is from Kunen's set theory textbook (\cite{kunen}*{p.~63}) and the second part is taken from Hrbacek-Jech's set theory textbook \cite{hrbacekjech}, where the term `countable' is used instead of `strongly countable'.  According to Veldman (\cite{veldje2}*{p.\ 292}), Brouwer studies set theory based on injections in \cite{brouwke}.
In this paper, `strongly countable' and `countable' shall exclusively refer to Definition~\ref{standard}.  

\smallskip

Thirdly, the notion of \emph{bounded variation} (abbreviated $BV$) was first explicitly\footnote{Lakatos in \cite{laktose}*{p.\ 148} claims that Jordan did not invent or introduce the notion of bounded variation in \cite{jordel}, but rather discovered it in Dirichlet's 1829 paper \cite{didi3}.} introduced by Jordan around 1881 (\cite{jordel}) yielding a generalisation of Dirichlet's convergence theorems for Fourier series.  
Indeed, Dirichlet's convergence results are restricted to functions that are continuous except at a finite number of points, while functions of bounded variation can have (at most) countable many points of discontinuity, as already studied by Jordan, namely in \cite{jordel}*{p.\ 230}.
Nowadays, the \emph{total variation} of $f:[a, b]\di \R$ is defined as:
\be\label{tomb}\textstyle
V_{a}^{b}(f):=\sup_{a\leq x_{0}< \dots< x_{n}\leq b}\sum_{i=0}^{n} |f(x_{i})-f(x_{i+1})|.
\ee
If this quantity exists and is finite, one says that $f$ has bounded variation on $[a,b]$.
Now, the notion of bounded variation is defined in \cite{nieyo} \emph{without} mentioning the supremum in \eqref{tomb}; this approach can also be found in \cites{kreupel, briva, brima}.  
Hence, we shall distinguish between the following two notions.  As it happens, Jordan seems to use item \eqref{donp} of Definition \ref{varvar} in \cite{jordel}*{p.\ 228-229}, providing further motivation for the functionals introduced in Definition \ref{JDR}.
\bdefi[Variations on variation]\label{varvar}
\begin{enumerate}  
\renewcommand{\theenumi}{\alph{enumi}}
\item The function $f:[a,b]\di \R$ \emph{has bounded variation} on $[a,b]$ if there is $k_{0}\in \N$ such that $k_{0}\geq \sum_{i=0}^{n} |f(x_{i})-f(x_{i+1})|$ 
for any partition $x_{0}=a <x_{1}< \dots< x_{n-1}<x_{n}=b  $.\label{donp}
\item The function $f:[a,b]\di \R$ \emph{has {a} variation} on $[a,b]$ if the supremum in \eqref{tomb} exists and is finite.\label{donp2}
\end{enumerate}
\edefi
This definition suggests a three-fold variation for any operation on functions of bounded variation, namely depending on whether this operation has access to the supremum \eqref{tomb}, an upper bound on the supremum \eqref{tomb}, or none of these. 
The fundamental theorem about $BV$-functions (see e.g.\  \cite{jordel}*{p.\ 229}) is as follows.
\begin{thm}[Jordan decomposition theorem]\label{drd}
A function $f : [0, 1] \di \R$ of bounded variation is the difference of two non-decreasing functions $g, h:[0,1]\di \R$.
\end{thm}
The computational properties of Theorem \ref{drd} have been studied extensively via second-order representations in e.g.\ \cites{groeneberg, kreupel, nieyo, verzengend}.
The same holds for constructive analysis by \cites{briva, varijo,brima, baathetniet}, involving different (but related) constructive enrichments.  
Now, arithmetical comprehension suffices to derive Theorem \ref{drd} for various kinds of second-order \emph{representations} of $BV$-functions in \cite{kreupel, nieyo}, i.e.\ finite iterations of the Turing jump suffice to compute the associated Jordan decomposition.  
By contrast, the results in \cite{dagsamXII} imply that no functional $\SS_{k}^{2}$ can compute the Jordan decomposition $g, h$ from $f$ in Theorem~\ref{drd} in general.  

\smallskip

Fourth, $BV$-functions are \emph{regulated} (called `regular' in \cite{voordedorst}), i.e.\ for every $x_{0}$ in the domain, the `left' and `right' limit $f(x_{0}-)=\lim_{x\di x_{0}-}f(x)$ and $f(x_{0}+)=\lim_{x\di x_{0}+}f(x)$ exist.  
Scheeffer studies regulated functions in \cite{scheeffer} (without using the term `regulated'), while Bourbaki develops Riemann integration based on regulated functions in \cite{boerbakies}.  
Weierstrass' `monster' function is a natural example of a regulated function not of bounded variation.  
As a testimony to its robustness, many functionals from the $\Omega$-cluster
remain in this cluster if we extend their input domain from `bounded variation' to `regulated'; the same holds for any intermediate class, as discussed in more detail in Remark \ref{essenti}.

\smallskip

Fifth, a somewhat unexpected result regarding $BV$-functions is the \emph{Banach indicatrix theorem} from \cite{banach1}, implying that for a continuous $BV$-function $f:[a,b]\di \R$, we may compute \eqref{tomb} as follows:
\be\label{indix}\textstyle
V_{a}^{b}(f)=\int_{\R} N(f)(y) dy, \textup{ where $N(f)(y)=\# \{x\in [a,b ]: f(x)=y\}$}.
\ee
The function $N(f)$ is called the \emph{Banach indicatrix} of $f$ on $[0,1]$.
By the (Russian language) results in \cite{lovely, lovely2}, \eqref{indix} also holds for regulated functions, as explained in \cite{voordedorst}*{p.\ 44} in English.

\smallskip

Sixth, the {Jordan decomposition theorem} as in Theorem \ref{drd} shows that a $BV$-function can be `decomposed' as the difference of monotone functions. 
This is however not the only result of its kind: Sierpi\'{n}ski's establishes in \cite{voordesier} that for a regulated function $f:[0,1]\di \R$, there are $g, h$ such that $f=g\circ h$ with $g$ continuous and $h$ strictly increasing on their respective domains. 
There are a number of similar results, to be found in \cite{voordedorst}*{p.\ 91}.

\smallskip

Seventh, an important sub-class of $BV$ is the Sobolev space $W^{1, 1}$, which consists of those $L_{1}$-functions with \emph{weak derivative} in $L_{1}$ (see Section \ref{flukkkk} for definitions).
As a testimony to its robustness, many functionals from the $\Omega$-cluster
remain in this cluster if we \emph{restrict} their input domain from `bounded variation' to `in $W^{1, 1}$'.

\smallskip

Eighth, we shall study the following notions of weak continuity.  
\bdefi\label{flung} For $f:[0,1]\di \R$, we have the following definitions:
\begin{itemize}
\item $f$ is \emph{upper semi-continuous} at $x_{0}\in [0,1]$ if $f(x_{0})\geq_{\R}\lim\sup_{x\di x_{0}} f(x)$,
\item $f$ is \emph{lower semi-continuous} at $x_{0}\in [0,1]$ if $f(x_{0})\leq_{\R}\lim\inf_{x\di x_{0}} f(x)$,
\item $f$ is \emph{quasi-continuous} (resp.\ \emph{cliquish}) at $x_{0}\in [0, 1]$ if for $ \epsilon > 0$ and an open neighbourhood $U$ of $x_{0}$, 
there is a non-empty open ${ G\subset U}$ with $(\forall x\in G) (|f(x_{0})-f(x)|<\eps)$ (resp.\ $(\forall x, y\in G) (|f(x)-f(y)|<\eps)$).
\end{itemize}
\edefi
Semi-continuity goes back to Baire and quasi-continuity goes back to Volterra (see \cite{beren2}).  
It is known that the sum of two cliquish functions is cliquish, while the sum of quasi-continuous functions is only cliquish in general (\cite{quasibor2}).
We note that regulated functions are cliquish everywhere but not necessarily quasi-continuous everywhere, i.e.\ continuity notions weaker than quasi-continuity
do not seem to provide extra information about regulated functions.  

\smallskip

Finally, related to semi-continuity, we can define the \emph{lower and upper envelope} $\underline{f}$ and $\overline{f}$, say for any bounded function $f:\R\di \R$, as follows:
\be\label{confaged}
\underline{f}(y):= \sup_{\delta>0}\inf_{|x-y|<\delta}f(x) \textup{ and } \overline{f}(y):= \inf_{\delta>0}\sup_{|x-y|<\delta}f(x),
\ee
which satisfy $\underline{f}\leq f\leq \overline{f}$ and are respectively lower and upper semi-continuous.
Intuitively, $\underline{f}$ (resp.\ $\overline{f}$) is the supremum (resp.\ infimum) of all lower (resp.\ upper) semi-continuous dominated by $f$ (resp.\ that dominate $f$). 
One readily shows that finding \eqref{confaged} in general is computationally equivalent to $\exists^{3}$.

\section{A $\lambda$-calculus formulation of Kleene's computation schemes}\label{maink}
\subsection{Introduction: total versus partial objects}
Turing's `machine' model (\cite{tur37}) captures \emph{computing with real numbers} and -by its very nature- involves partially defined objects.  
By contrast, Kleene's S1-S9 (\cite{kleeneS1S9}) is an extension of the Turing approach meant to capture \emph{computing with objects of any finite type}, where
everything is always assumed to be \emph{total}.  In this section, we introduce an equivalent $\lambda$-calculus formulation of S1-S9 (see Section \ref{kleedef}) that accommodates 
partial objects.  Our motivation for this new construct is as follows.
\begin{itemize}
\item We have previously hand-waved the extension of S1-S9 to partial objects, leading to confusion among our readers regarding certain technical details.  
\item Proofs are a lot more transparent based on fixed points (from the $\lambda$-calculus) instead of the recursion theorem (hardcoded by S9; see Section~\ref{kle9}).
\item The functional $\Omega_{1}$ from Definition \ref{defomega1} is a natural object of study that is \emph{fundamentally partial} in nature.  
Indeed, we show that $\Omega_{1}$ is not computationally equivalent to any \emph{total} functional (Section \ref{kepl}).  
To this end, we extend the notion of \emph{countably based functional} to partial objects.  
\item The functional $\Omega$ from Definition \ref{defomega} is a natural object of study that is \emph{fundamentally partial} in nature.  
Indeed, we show that $\Omega$ is not computationally equivalent to any \emph{total} functional (Section \ref{kepl2}).  
To this end, we introduce a general way of approximating computations in our new model.  
\end{itemize}
Finally, as an application of our newly minted framework, we can show that $\Omega_{1}$ 
and the Suslin functional $\SS^{2}$ compute the halting problem for \emph{infinite time} Turing machines (Theorem \ref{jable}) in Section \ref{cross}.
\subsection{Definition of the $\lambda$-calculus formulation of S1-S9}\label{kleedef}
In this section, we introduce a new version of the $\lambda$-calculus (Section \ref{subseq2.3}) and show that this computational model is equivalent to S1-S9 (Section \ref{sec3}).  
We first introduce some basic definitions (Section \ref{bade}) and the language of our new model (Section \ref{subseq2.2}).

\subsubsection{Some basic definitions}\label{bade}
We introduce the usual notions of rank and type.
\begin{definition}\label{def1}{\em
The \emph{types of finite rank} are inductively defined as follows:
\begin{itemize}
\item $0$ is a type with rank $\ra(0) = 0$,
\item if $\sigma_1, \ldots,\sigma_n$  are types, then $\sigma = (\sigma_1 , \ldots , \sigma_n \rightarrow 0)$ is a type, with rank $\ra(\sigma) = \max\{\ra(\sigma_i) +  1: i = 1 , \ldots , n\}$
\end{itemize}
We let $\Ty(k)$ be the set of types of rank $\leq k$}
\end{definition}
We will mainly be concerned with $\Ty(3)$, but develop some concepts for the full type hierarchy $\Ty=\cup_{k}\Ty(k)$. We adopt the following standard convention.
\begin{nota}[Brackets and arrows]\rm
For $n\geq 2$, we (may) rewrite $(\sigma_1 , \ldots , \sigma_n \rightarrow 0)$ to $(\sigma_1 \rightarrow (\sigma_2 , \ldots , \sigma_n\rightarrow 0))$ and also iterate this typing convention.  
We (may) also drop the brackets `$($' and `$)$' when this does not create too much ambiguity.
\end{nota}
As detailed in Definition \ref{def2}, we interpret $\sigma\in \Ty$ as two sets $\F(\sigma)$ and $\Pa(\sigma)$, where the former is a subset of the latter. 
This definition reflects our aim, namely to find a conceptually simpler approach to S1-S9.   
Indeed, our definition allows computations relative to a partial functional, but we never apply (even a partial) functional to a non-total functional; we also do not restrict to the pure types. 

\smallskip

Now, the extension from pure to finite types is in general just a matter of convenience.  By contrast, we show in  Section \ref{kepl} that the extension to partial functionals (only accepting total inputs) is \emph{essential}. 
 Indeed, we identify important functionals based on basic theorems from mainstream mathematics, including $\Omega_{1}$ and $\Omega$ from Definitions \ref{defomega1} and \ref{defomega}, that cannot be represented by total functionals.  

\begin{definition}[Types]\label{def2} 
{Define $\N_\bot = \N \cup  \{\bot\}$ where $\bot$ is a symbol for \emph{undefined}.  
\begin{itemize}
\item For $\sigma\in \Ty$, define the set $\F(\sigma)$ as follows:
\begin{itemize}
\item $\F(0) = \N$,
\item for $\sigma = (\sigma_1 , \ldots , \sigma_n \rightarrow 0)$, $\F(\sigma)$ is the set of all functionals of type $\F(\sigma_1) \times \cdots \times \F(\sigma_n) \rightarrow \N$.
\end{itemize}
\item For $\sigma\in \Ty$, define the set $\Pa(\sigma)$ as follows:
\begin{itemize}
\item $\Pa(0) = \N_\bot $,
\item for $\sigma = (\sigma_1 , \ldots , \sigma_n \rightarrow 0)$, $\Pa(\sigma)$ is the set of all functionals of type $\F(\sigma_1) \times \cdots \times \F(\sigma_n) \rightarrow \N_\bot$.
\end{itemize}
\end{itemize}
}\end{definition}
Since we are mainly interested in functionals originating from ordinary mathematics, we generally restrict our attention to statements of the form
\be\label{zrux}
(\forall x^\sigma)( \exists y^\tau)\big[{\Gamma}(x) \rightarrow {\Delta}(x,y)\big],
\ee
where $\ra(\sigma) = 2$, $\ra(\tau) \leq 2$ and where $\Delta$ is -in some intuitive sense- simple but $\Gamma$ can be complicated or complex.  

\smallskip

A \emph{realiser} for \eqref{zrux} is in general a partial functional $\Phi$ with domain $\{x \in \F(\sigma) : {\Gamma}(x)\}$ and image in $\{y\in \F(\tau):{\Delta}(x,\Phi(y))\}$. 
Thus, we only study total objects of type $\sigma$ in case $\ra(\sigma) \leq 2$, and only certain kinds of partial objects at level 3 are needed for analysing realisers.

\smallskip

As suggested above, we wish to replace Kleene's S9 by a $\lambda$-calculus construct namely a \emph{least fixed point} operator.  For the latter, we also need partial objects when describing the actual computations, which is one motivation for introducing the sets $\Pa(\sigma)$ for all types $\sigma\in \Ty$ in Definition \ref{def2}. 
\subsubsection{The language of calculations}\label{subseq2.2}
In this section, we introduce the term language used to express computable functions and functionals. 

\smallskip

Our language $\mathcal L$ is modelled on Plotkin's PCF (\cite{plotje}) with few and basic modifications. Plotkin's PCF is based on Scott's LCF (\cite{scottfree}), which in turn is inspired by Platek's PhD thesis (\cite{pphd}).
The use of least fixed points as the core concept of general computability theory has also been advocated by e.g.\ Moschovakis in \cite{YannisII}.
For some background from the perspective of the foundations of computer science, we refer to Streicher's monograph \cite{streikersc}*{\S1-3}.

\smallskip

First of all, we define the constants (Definition \ref{definition8}) and typed terms (Definition~\ref{cringe}) of our language $\mathcal L$.
\begin{definition}\label{definition8}
{\em
The \emph{primary} constants of the language $\mathcal L$ are:
\begin{multicols}{2}
\begin{itemize}
\item $\hat 0$ of type 0,
\item ${\suc}$ of type $0 \rightarrow 0$,
\item${\pd}$ of type $0 \rightarrow 0$,
\item ${\case}$ of type $0,0,0 \rightarrow 0$.
\end{itemize}
\end{multicols}
}\end{definition}
These constants will be interpreted as total objects, $\hat 0$ as 0, ${\suc}$ as the successor-function on $\N$, ${\pd}$ as the predecessor function (with ${\pd}(0) = 0$) on $\N$ and ${\case}(0,x,y) = x$ while ${\case }(z+1,x,y) = y$ for all $x,y,z$ in $\N$. For $\case(z,x,y)$ to take a defined value, we require that $z$ is total, while only one of the $x$ and $y$ needs to have a defined value.
\begin{definition}\label{cringe}
{\em The typed terms in $\mathcal L$ are defined by recursion as follows.
\begin{itemize}
\item All primary constants are terms in $\mathcal L$.
\item For $\sigma \in \Ty(3)$, there is an infinite list $x_i^\sigma$ of variables of type $\sigma$, which are terms in $\mathcal L$.
\item For $\sigma \in \Ty(3)$ and $\phi \in \Pa(\sigma)$, there is a \emph{secondary} constant $\hat \phi$, which we often just denote as $\phi$. 
The object $\hat \phi$ is a term in $\mathcal L$ and is interpreted as $\phi$.
\item If $t$ is a term of type $\sigma \rightarrow \tau$ and $s$ is a term of type $\sigma$, then $(ts)$ is a term of type $\tau$. This is interpreted as the application of the interpretation of $t$ to the interpretation of $s$.
\item If $\ra(\sigma \rightarrow \tau) \leq 3$ (and hence $\sigma \in \Ty(2)$) and $t$ is a term of type $\tau$, then $(\lambda x_i^\sigma t)$ is a term of type $\sigma \rightarrow \tau$. This is interpreted as standard $\lambda$-abstraction.  
\item If $\ra(\sigma) \leq 3$ and $t$ is a term of type $\sigma$, then $(\mu x_i^\sigma t)$ is a term of type $\sigma$. 
The interpretation of this case is discussed below.
\end{itemize}}\end{definition}
%
The set $\Pa(\sigma)$ has a canonical partial ordering `$\preceq_\sigma$' representing extensions of partial functionals, i.e.\ `$\phi^{\sigma}\preceq_{\sigma}\psi^{\sigma}$' means that the domain of $\psi$ includes that of $\phi$.  
The reason for introducing the partial orders $(\Pa(\sigma),\preceq_\sigma)$ is that they can be used for interpreting self-referential programs, like we intend to do with our $(\mu x_i^\sigma t)$-terms.
An essential notion here is \emph{monotonocity}, defined as follows.
\bdefi\label{floeperpoepie}
Any $\phi^{\sigma\di \tau}$ is \emph{monotone} if $(\forall x^{\sigma}, y^{\sigma} )(x\preceq_{\sigma} y \di \phi(x)\preceq_{\tau}\phi(y)  )$.
\edefi
We now introduce the interpretation of all typed terms, as monotone functions, by recursion on the term. 
We adopt the standard approach, i.e.\ we consider a term $t$ of type $\sigma$ and a set of variables $x_1^{\sigma_1}, \ldots , x_n^{\sigma_n}$ that contains all variables free or bounded in $t$, and then interpret $t$ as an increasing map 
\[
[[t]]:\Pa(\sigma_1) \times \cdots \times \Pa(\sigma_n) \rightarrow \Pa(\sigma). 
\]
In Definition \ref{intdef}, we let $\phi_i$ be the argument in $\Pa(\sigma_i)$ and we drop the upper (type) index when not needed.
Note that `monotone' always refers to Definition \ref{floeperpoepie}.
\begin{definition}[Interpretation of terms]\label{intdef}{\em ~
\begin{itemize}
\item For variables $x_{i}$, we put $[[x_i]](\phi_1 , \ldots , \phi_n) = \phi_i$.
\item If $\hat \phi$ is a secondary constant for $\phi$, we let $[[\hat \phi]](\phi_1 , \ldots , \phi_n) = \phi$.
\item The basic objects $\hat 0$, ${\suc}$, ${\pd}$, and ${\case}$ are treated in the same way as secondary constants (see Definition \ref{cringe}).
\item If $t = (uv)$ and $v$ has type $\tau$, we let
\[
[[t]](\phi_1 , \ldots , \phi_n) = [[u]](\phi_1 , \ldots , \phi_n)([[v]](\phi_1 , \ldots , \phi_n)).
\]
which is \emph{everywhere} undefined \textbf{unless} $[[v]](\phi_1 , \ldots , \phi_n) \in \F(\tau)$.
\item 
If $t = (\lambda x_i s)$ and $\phi' \in \F(\sigma_i)$, we let 
\[ 
[[t]](\phi_1 , \ldots , \phi_n)(\phi') = [[s]](\phi_1 , \ldots , \phi_{i-1} , \phi',\phi_{i+1} , \ldots , \phi_n).
\]
If $\phi'$ is (only) partial, we let the value be undefined.
\item If $t = (\mu x_i s)$, and $i = 1$, fix $\phi_2 , \ldots , \phi_n$ and note that $\Phi(\psi) = [[s]](\psi , \phi_2 , \ldots , \phi_n)$ is monotone in $\psi$ and has a least fixed point $\psi_\infty$.
We put $[[t]](\phi_1 , \ldots , \phi_n) = \psi_\infty$.  The general case is treated in the same way, but is much more tedious. 
\end{itemize}
Note that in the cases where a variable is bounded, the corresponding argument will remain an argument of the interpretation, now as a dummy one.
The existence of a (least) fixed point in the final item of Definition \ref{intdef} is guaranteed by the well-known Knaster-Tarski theorem.
}\end{definition}
Next, we come to the crucial notion of `computable in', where we recall the ordering `$\preceq_{\sigma}$' on $\Pa(\sigma)$ introduced after Definition \ref{cringe}.
\begin{definition}[Computability]\label{def.comp}{\em For $\sigma_1 , \ldots , \sigma_n , \delta\in \Ty(3)$, let $\phi_1 , \ldots , \phi_n,\psi$ be partial objects of the corresponding types. 
We say that 
\[
\text{$\psi$ \emph{is computable in} $\phi_1 , \ldots , \phi_n$}
\]
if there is a term $s$ of type $\tau$,  with free variables $x_1 , \ldots , x_n$ of types $\sigma_1 , \ldots , \sigma_n$, and without any secondary constants, such that 
\[  
\psi \preceq_\tau [[s(x_1 , \ldots , x_n / \hat \phi_1 , \ldots , \hat \phi_n)]] .
\]
}\end{definition}
From now on, we shall drop all hats `\^{}' when discussing computability, unless this creates ambiguity.  We also follow the standard convention for writing iterated applications as $t_1t_2 \cdots t_n$, meaning $(\cdots (t_1t_2)t_3\cdots t_n)$.
\subsubsection{Computation trees and an operational-like semantics}\label{subseq2.3}
In this section, we introduce the crucial notion of \emph{computation tree} and associated semantics.

\smallskip

We let   $t$ be a closed term of type 0 and we assume that $[[t]] \in \N$.  The definition of $[[t]]$ in Definition \ref{intdef} is by recursion on the term $t$, and for some cases, by a transfinite sub-induction that is monotone. 
As a consequence, there is some kind of well-founded tree witnessing that $[[t]] = n$.  
We define one such `computation tree' below.  Our definition is inspired by Plotkin's operational semantics for PCF (\cite{plotje}), but since we are modelling infinitary computations, our modifications to Plotkin's approach are far-reaching. 

\smallskip

Another motivation for computation trees is that we want to recapture the qualities of computation trees defined for Kleene's S1-S9. 
In order to obtain a true analogue of the latter, we need to perform a little bit of coding of the elements in the tree, as will become clear below.  

\smallskip

For the definition of computation tree, we use the fact that all terms of type $0$ can be written as an iterated application $t = t_1 \cdots t_n$ where $t_1$ is not itself an application term.
For some terms $t$, the computation tree $T[t]$ will be assigned a value, namely an integer $a$. 
The definition of $T[t]$ is top-down, while the definition of the value is bottom-up. The whole construction can be viewed as a positive inductive definition, in (intended) analogy with the set of Kleene computations.
\begin{definition}[Computation tree]\label{definition12}{\em
Let $t$ be a closed term of type $0$ with parameters $\Phi_1, \ldots ,  \Phi_n$. We define the \emph{computation tree} $T[t]$ of $t$, consisting of sequences of terms, by recursion on the complexity of $t$ as follows. 
\begin{enumerate}
\renewcommand{\theenumi}{\roman{enumi}}
\item We let the empty sequence be the root of each $T[t]$, and we identify a term $t$ with the corresponding sequence of length 1.\label{kirst}
\item If $t = 0$, then $t$ is the only extra node in  $T[t]$ and $0$ is the value of $T[t]$.
\item If $t = \suc\; t_1$, then we concatenate $t$ with the sequences in  $T[t_1]$. If $T[t_1]$ has the value $a$, then $T[t]$ has the value $a+1$.
\item If $t = \pd \;t_1$, then we act in analogy with the case above.
\item If $t = \case \;t_1t_2t_3$, then $T[t]$ consists of $t$ concatenated with all sequences in $T[t_1]$ and, if $T[t_1]$ has a value $a$, the sequences in $T[t_2]$ or $T[t_3]$ depending on $a$. The value of $T[t]$ will then be the corresponding value of $T[t_2]$  or $T[t_3]$ if it exists.\label{kast}
\item If $t = \Phi_i t_2 \cdots t_n$, then $t_2$ is of type $\tau = (\delta_1 , \ldots , \delta_m \rightarrow 0)$.
Then $T[t]$ is $t$ concatenated with the sequences in the following trees $T[s]$ and $T[s']$.
 \begin{itemize}
  \item For $(\phi_1 , \ldots , \phi_m) \in \F(\delta_1) \times \cdots \times \F(\delta_m)$, we concatenate with the sequences in  $T[s]$ where $s = t_2\phi_1 \cdots \phi_m$.
 \item If each such  $T[s]$ has a value, these values define a functional $\Psi_2$. In this case, define $\xi: = \Phi_i(\Psi_2)$ and define $s'$ as the term $\xi t_3 \cdots t_n$. \label{bwist}
 \item We let $T[t]$ contain $t$ concatenated with all sequences in $T[s']$, and if $T[s']$ has a value, we let this be the value of $T[t]$.
 \end{itemize}
 \item If $t =( \lambda x t_1)t_2 \cdots t_n$ we first construct $T[t]$ via the trees $T[s]$,  as in the previous item, and if all sub-trees have values, we get the  total functional $\Psi_2$ as above. We then consider the term
 \[
 s' = t_1[x/t_2]t_3 \cdots t_n
 \] 
 and concatenate $t$ with all sequences in this $T[s']$ as well. The value of this $T[s']$ will then, if it exists, be the value of $T[t]$.\label{twist}
 \item If  $t = (\mu x t_1)t_2 \cdots t_n$, put $s = t_1[x/(\mu x t_1)]t_2 \cdots  t_n$. Then $T[t]$ consists of $t$ concatenated with all sequences in $T[s]$. If $T[s]$ has a value, then the latter is also the value of  $T[t]$.\label{twister}
\end{enumerate}
}\end{definition}
\begin{remark}[Conversion by any other name]\label{xx}{\em For item \eqref{twist} in Definition \ref{definition12}, we have made a twist to the standard format of $\beta$-conversion, that $(\lambda x t)s$ directly converts to $t[x/s]$. The reason is that since we can only have that an applicative term gives a value (different from $\bot$) if the argument is a total functional, classical $\beta$-conversion does not respect the semantics of the terms. We refer the conversion implicit in item \eqref{twist} as \emph{modified $\beta$-conversion}.
}\end{remark}
It is clear that for $T[t]$ as in Definition \ref{definition12} to have a value, the tree must be well-founded, since a value always depends on the values of the sub-trees, except for the ground term $0$. The converse is not always true, since an application involving partial functionals may yield a well-founded tree without a value.

\smallskip

The main result of this section is the following theorem, which we prove through a series of lemmas.
\begin{theorem}\label{thm.13} Let $t$ be a closed term of type $0$ as above and let $n \in \N$. The following are equivalent:
\begin{itemize}
\item the interpretation of $t$ satisfies $[[t]] = n$
\item the tree $T[t]$ is well-founded and with  value $n$.
\end{itemize}
\end{theorem}
We first establish the `easy' direction. 
\begin{lemma}\label{lemma14}
The upward implication in Theorem \ref{thm.13} holds.
\end{lemma}
\begin{proof}
We prove this by induction on the ordinal rank of $T[t]$. 
We provide a proof by cases following the items of Definition \ref{definition12}.  
All cases except items \eqref{twist} and \eqref{twister} are straightforward, and are left for the reader. 

\smallskip

For item \eqref{twist},  since each $T[s]$ is a genuine sub-tree of $T[t]$, we obtain from the induction hypothesis that $\Psi_2 = [[t_2]]$. We then apply the induction hypothesis to $s'$ to obtain the claim.

\smallskip

For item \eqref{twister}, let $t = (\mu x t_1)t_2 \cdots t_n$ and assume that $T[t]$ is well-founded, and has a value $a$. Then $T[s]$ is well-founded, where
$s = t_1[x/(\mu x t_1)]t_2 \cdots  t_n$, and also has value $a$.  By the induction hypothesis, we obtain
\[
[[t_1[x/(\mu x t_1)]t_2 \cdots  t_n]] = a.
\]
Since $[[(\mu x t_1)]] = [[t_1[x / (\mu x t_1)]]$, it follows that $[[(\mu x t_1)t_2 \cdots t_n]] = a$. 
\end{proof}
For the downward implication in Theorem \ref{thm.13}, we need some technical machinery, including a way of handling substitutions.   

\smallskip

Let $t$ be a term containing secondary constants $ \vec \phi$, let $\vec s$ be a sequence of closed terms with types matching those of  $\vec \phi$ and where each $\phi_i$  is a sub-function of $[[s_i]]$.  
We write $t[\vec \phi/\vec s]$ for the term where all occurrences  of each $\phi_i$ are replaced by $s_i$.
By abuse of notation, $\preceq_{\sigma}$ is also the canonical ``more defined than"-ordering on $\Pa(\sigma)$.
\begin{definition}{\em \label{definition15}~
\begin{itemize}
\item If $t$ is a closed term of type $0$ such that $T[t]$ is well-founded and has a value $a$, we let $[[t]]_w  = a$.

\item If $t$ is a closed term of type $\delta_1 , \ldots , \delta_m$ with $m \geq 1$, we put 
\[
[[t]]_w =( \lambda (\xi_1 , \ldots , \xi_m) \in \F(\delta_1) \times \cdots \times \F(\delta_m) ) [[t\xi_1 \cdots \xi_m]]_w.
\]
\end{itemize}}
\end{definition}
\noindent
The previous definition enables us to prove the following. 
\begin{lemma}{\bf (Substitution Lemma)}
Let $t$ be a term of type 0 with secondary constants among $\phi_1 , \ldots , \phi_n$ and such that $\phi_i \preceq [[s_i]]_w$ for each $i = 1 , \ldots , n$. Then the interpretations satisfy $[[t]]_w \preceq [[t[\vec \phi / \vec s]\/]]_w$. \end{lemma}
\begin{proof}
The proof is by induction on the maximal rank of the $\phi_i$ and a sub-induction on the number of $\phi_i$ of this maximal rank.
We will substitute $s_i$ for $\phi_i$ in sub-terms of $t$ by a bottom-up procedure; the only nodes where we need to perform this substitution, are nodes of the form $\phi_i v_1 \ldots v_l$.  

\smallskip

As an  hypothesis in the sub-sub-induction on the rank of the computation tree, we can assume that each $v_j$ is already the result of the substitution, that is, they are without any occurrences of $\phi_i $. We can further assume, as a part of the same induction hypothesis,  that each $[[v_j]]_w$ is total. We must however show that we can replace $\phi_i$ with $s_i$ and still have a well-founded computation tree with a value.

\smallskip

By assumption we have that $\phi_i \xi_1  \cdots \xi_l = [[s_i \xi_1 \cdots \xi_l]]_w$ for all total $\xi_1 , \ldots , \xi_l$ of the appropriate types.
If we now fix $\xi_j = [[v_j]]_w$, then by the main induction hypothesis we have that  
\[ 
[[\phi_iv_1 \cdots v_l]]_w = [[\phi_i \xi_1 \cdots \xi_l]]_w \preceq [[s_i\xi_1 \ldots \xi_l]]_w \preceq [[s_i v_1 \cdots v_l]]_w,
\] 
since the ranks of the types of the $\xi_j$ are lower than the rank of $\phi_i$.  Hence, we are free to make the aforementioned substitution. 
\end{proof}
\begin{definition}[Normal terms]{\em 
 Let $t$ be a closed term with secondary constants $ \phi_1, \cdots ,\phi_n$, where the type of $\phi_i$ is $\tau_{i,1} , \ldots , \tau_{i,m_i}\rightarrow 0$.
\begin{itemize}
\item If $t$ is of type 0, we call $t$ \emph{normal} if \[  [[\;t\;]] \preceq [[\;t[\phi_1 , \ldots , \phi_n/s_1 , \ldots , s_n]\;]]_w \] whenever $\phi_i \preceq [[s_i]]_w$ for each $i = 1 , \ldots , n$. 
\item If $t$ is of type $\sigma_1, \ldots , \sigma_k \rightarrow 0$ we call $t$ \emph{normal} if the term $tt_1 \cdots t_k$ is normal whenever $t_1, \ldots , t_k$ are normal, and of types $\sigma_1 , \ldots , \sigma_k$.
\end{itemize}
}\end{definition}
\begin{lemma}\label{lemma16} All closed terms are normal. \end{lemma}
\begin{proof}We prove this by induction on the term $t$, and the proof is split into cases following the inductive definition of terms as in Definition \ref{cringe}. Throughout, we write $\vec \phi$ for $\phi_1 , \ldots , \phi_n$ and $\vec s$ for a corresponding sequence $s_1 , \ldots , s_n$ satisfying the assumption. This will be the convention whenever we need to prove that a term of type 0 is normal.  
The proof proceeds based on the following steps. 
\begin{itemize}
\item The case $t = 0$ is clearly trivial.
\item The cases $t = \suc$, $t = \pd$, and $t = \case$. These cases are trivial by the induction hypothesis. For the sake of completeness, consider $\suc$. We must prove that $\suc\;t$ is normal if $t$ is normal. By the assumption $[[t]] \preceq [[t[\vec \phi/\vec s]\;]]_w$, and the same relation will be preserved for $\suc\; t$.
\item The case $t = \Phi $ where $ \Phi$ is a secondary constant. We must prove that if $t_1 , \ldots , t_k$ are normal, then $\Phi t_1 \cdots t_k$ is normal, i.e.\ we must prove that if $\Phi \preceq [[s]]_w$, then $[[\Phi t_1 \cdots t_n]] \preceq [[s t_1[\vec \phi /\vec s]  \cdots t_k[\vec \phi / \vec s]\;]]_w$. If the first value is $\bot$, there is nothing to prove. If not, each $[[t_i]]$ is a total object $\xi_i$ and we can then use the substitution lemma on $\Phi \xi_1 \cdots \xi_n$, where we substitute $s$ for $\Phi$ and $t_i[\vec \phi/\vec s]$ for $\xi_i$.
\item The application case $t = (t_1t_2)$. By the induction hypothesis, both $t_1$ and $t_2$ are normal. Then, if $t_3 , \ldots , t_k$ are normal, we will have, since $t_1$ is normal, that $(t_1t_2)t_3 \cdots t_k = t_1t_2 \cdots t_k$ is normal. 

\item The abstraction case $t = (  \lambda x t_1)$. By the induction hypothesis, $t_1$ is normal.  Let $t$ be of type $\tau, \tau_2, \ldots , \tau_k \rightarrow 0$, that is, $x$ is of type $\tau$ and $t_1$ is of type $\tau_1  =  \tau_2 , \ldots , \tau_k \rightarrow 0$.
We must prove that whenever $t_2 , \ldots ,,t_k$ are normal, and  given $\vec \phi$ and $\vec s$ as above, we have that \[ [[tt_2\cdots t_k]]\preceq [[ t[\vec \phi/\vec s]t_2[\vec \phi/\vec s]  \cdots t_k[\vec \phi/\vec s]\;]]_w. \]
Let $\phi = [[t_2[\vec \phi/\vec s]\;]]$ and $s = t_2[\vec \phi/\vec s]$. We use that $t_1$ is normal, and obtain
\begin{align*}
[[tt_2\cdots t_k]] = [[t_1[x / [[t_2]]]t_3 \cdots t_k]] 
&\preceq [[t_1[x,\vec \phi / s,\vec s]]_w \\
& = [[t[\vec \phi/\vec s]t_2[\vec \phi/\vec s] \cdots t_k[\vec \phi/\vec s]\;]]_w.
\end{align*} 
The inequality follows from the induction hypothesis and the final equality follows from \emph{modified} $\beta$-conversion (see Remark \ref{xx}).
\item The case $t = (\mu x t_1)$. By the induction hypothesis, $t_1$ is normal. Let $\vec \phi$ and $\vec s$ be as before, let $t_1$ be of type $\tau_1 = \tau_2 , \ldots , \tau_k \rightarrow 0$, the same type as $t$, and assume $(\xi_2, \ldots , \xi_k) \in \F(\tau_2)\times \cdots \times \F(\tau_k)$. Since the least fixed point $[[t]]$ is defined by a monotone induction, all $(\xi_2 , \ldots , \xi_k)$ with $\big([[t]]\xi_2 \cdots \xi_k\big) \in \N$ will be ranked by an ordinal. We will prove by induction on this rank that: 
\begin{center}
if $[[tt_2 \cdots t_k]] \in \N$, then  $[[t[\vec \phi/\vec s]t_2[\vec \phi/\vec s] \cdots t_k[\vec \phi/\vec s]\;]]_w=[[tt_2 \cdots t_k]] $.
\end{center}
Hence, let $\Psi_\alpha$ be the $\alpha$-th iteration of $[[t_1]]$, and assume $[[t_1]]\Psi_\alpha [[t_2]] \cdots [[t_k]] \in \N$. 
Now assume as (an) induction hypothesis that: 
\begin{center}
if $(\Psi_\alpha \xi_2 \cdots \xi_k) \in \N$, then $[[t[\vec \phi/\vec s]\xi_2 \cdots \xi_k]]_w =\Psi_\alpha \xi_2 \cdots \xi_k $.
\end{center}
This means that $\Psi_\alpha \preceq [[t[\vec \phi/\vec s]\; ]]_w$, and since 
$t_1$ is normal, we have that 
\be\label{lhs}
[[t_1[x/ \Psi_\alpha]t_2 \cdots t_k]] \preceq [[t_1[x,\vec \phi / t[\vec \phi/\vec s],\vec s]t_2[\vec \phi/\vec s] \cdots t_k[\vec \phi/\vec s]]_w.
\ee
Now, the left-hand side of \eqref{lhs} is $\Psi_{\alpha + 1}[[t_2]] \cdots [[t_k]]$ and the right-hand side of \eqref{lhs} is, by item \eqref{twister} of Definition \ref{definition12}, the following:
\[
[[t[\vec \phi/\vec s]t_2[\vec \phi/\vec s] \cdots t_k[\vec \phi/\vec s]\;]]_w.
\]  
Thus, the induction may continue. That the required order is preserved through limit ordinals is trivial. 
\end{itemize}
All cases having been treated, we are done.
\end{proof}
We now also have a proof of Theorem \ref{thm.13} as the upward direction is proved in Lemma \ref{lemma14} and the other direction is a special case of Lemma \ref{lemma16}.
\begin{remark}{\em The argument above is an adaption of standard proofs for normalisation theorems in typed $\lambda$-calculus, that the value of a term of type $0$ can be found through a finite iterated use of conversion rules. 
We did not use the restriction to $\Ty(3)$ in the proof above, but this restriction will give us that the trees $T(t)$, modulo a finite list of parameters, will be objects of type level 2.}\end{remark}
\subsubsection{Equivalence}\label{sec3}
We show that higher-order computability defined via Kleene's S1-S9  is equivalent to the computability model defined in the previous section, whenever this comparison makes sense. 

\smallskip

We will not introduce S1-S9 but instead refer to \cite{longmann}.  In particular, we only state the equivalence between S1-S9 and the computability model from the previous section and sketch the proof. 
Recall that the notion of pure type is as follows.
\bdefi[Pure types]
The type $0$ is a pure type.  If $k$ is a pure type, then $k+1 := k \rightarrow 0$ is a pure type.
\edefi
We (often) use the same notation for a pure type and its rank.
\begin{theorem}[Equivalence]\label{prize}
Let $k,k_1 , \ldots , k_n$ be pure types $\leq 3$ and let $\Phi, \Phi_1 , \ldots , \Phi_n$ be total elements of the respective types.
Then the following are equivalent:
\begin{itemize}
\item the functional $\Phi$ is computable in $\Phi_1 , \ldots , \Phi_n$ in the sense of Kleene's \emph{S1-S9},
\item there is a term $t$ of type $k$ with secondary constants among $\Phi_1 , \ldots , \Phi_n$ such that $\Phi = [[t]]$.
\end{itemize}
\end{theorem}
\begin{proof}(Sketch)
We prove the equivalence for $k = 3$; the other cases are proved in the same way. 
For the downward implication, we observe that the relation 
\[
\{e\}(\Phi_1 , \ldots , \Phi_n , F , \vec f , \vec a) \simeq b
\]
is defined via a monotone inductive definition. Coding sequences $\vec f$ in $\N^\N$ and sequences $\vec a$ in $\N$, we can see this as one definition of a partial functional $\Psi$ of type $0, 2,1,0 \rightarrow 0$ defined by one application of $(\mu x t)$ to a term $t$ with $\Phi_1 , \ldots , \Phi_n$ as parameters. From this, we can extract a term for $\Phi$ with $\Phi_1, \ldots , \Phi_n$ as parameters.  For the upward implication, we proceed by induction on the term $t$.   Coding sequences of arguments into one when needed, we can prove that $[[t]]$ is partially S1-S9-computable in the parameters. 
We then use the recursion theorem for S1-S9 to deal with $(\mu x t)$, and the rest of the cases are trivial. 
\end{proof}
We finish this section with another advantage of our newly minted computability model.  We will not explore this aspect in this paper (due to length isssues). 
\begin{rem}[Flexibility]\rm
One of the advantages of our alternative to Kleene's S1-S9 is that it is easy to restrict the model to natural sub-classes.  Indeed, one need only replace the fixed point operator $(\mu xt)$ by constants for some computable functionals. 
For instance, we may replace the fixed point operator by the `recursor constants' $\rec_\sigma$ of type $\sigma,(\sigma \rightarrow \sigma) \rightarrow (0 \rightarrow\sigma)$ and defined as follows:
\begin{itemize}
\item $\rec_\sigma(0,x,y) = x$,
\item $\rec_\sigma(k+1 , x , y) = y(\rec_\sigma(k,x,y))$.
\end{itemize}
Each $\rec_\sigma$ is readily seen to be computable and we obtain a computability-model based on the principles of G\"odel's $T$, but accepting partial arguments. 
\end{rem}

\subsubsection{Computation trees revisited}
In this section, we establish some technical results that make working with computation trees more straightforward.

\smallskip

In more detail, the trees $T[t]$ from Section \ref{subseq2.3} are essential for a finer analysis of computations.  A desirable feature of these trees is a representation as objects that are computable in the parameters involved. 
We can obtain such a result if we restrict attention to $\Ty(3)$, based on the observation that whenever we need to introduce new parameters in sub-computations, these parameters will be of type 0 or 1. %
For the rest of this section, we introduce the aforementioned desirable representation of computation trees (Definitions \ref{numb} and \ref{atrees}) and establish the required properties (Theorem \ref{flaheo}). 
\begin{definition}\label{numb}{\em
Let $t$ be a term with free variables, but without parameters. We let $\ll t\gg $ be the G\"odel-number of $t$ in some standard G\"odel-numbering. The number $\ll t \gg$ will contain information about  a set of variables $x_1 , \ldots , x_n$ containing all variables free or bounded in $t$, of their types, and how $t$ is syntactically built up.
 }\end{definition}

 The order in which we list variables or parameters is of course of no importance.  As a convention, the variables in $t[\vec \Phi,\vec f]$ are such that $\vec \Phi$ contains all the parameters of type rank two or three, while $\vec f$ consists of parameters of type rank  0 and 1. 
 
 \smallskip
 
 Given a term $t[\vec \Phi,\vec f]$ with a value in $\N$, we now define the \emph{alternative} {computation tree} $T^\ast[t;\vec \Phi]$ as a tree of sequences of objects of the form $\langle \ll s \gg , \vec g, a\rangle$ containing exactly the same information as $T\big[t\big[\vec \Phi,\vec f\big]\big ]$, now with the values given at each node.
 \bdefi[Computation trees bis]\label{atrees}
For $t[\vec \Phi , \vec g]$ a term of type 0 with a value, we define $T^*[t;\vec \Phi]$ by adopting items \eqref{kirst}-\eqref{kast} and \eqref{twister} from Definition \ref{definition12} together with the following alternative items.
 \begin{itemize}
 \item[(vi')] For the case $t[\vec \Phi , \vec g] = \Phi_i t_2[\vec \Phi , \vec g] \cdots t_n[\vec \Phi , \vec g]$, the term $t_2$ is of type $\tau = \delta_1 , \ldots , \delta_m \rightarrow 0$ 
 where each $\delta_j$ is of type ranks 0 or 1.
Then $T^*[t;\vec \Phi]$ consists of $\langle \ll t\gg, \vec g , [[t]] \rangle$ concatenated with the sequences in the trees $T^*[s;\vec \Phi]$ and $T^*[s';\Phi]$) as in the following steps.
 \begin{itemize}
  \item[-] Firstly, for each $(\phi_1 , \ldots , \phi_m) \in \F(\delta_1) \times \cdots \times \F(\delta_m)$, we concatenate with the sequences in  $T^*[s;\vec \Phi]$ where $s = t_2[\vec \Phi,\vec g]\phi_1 \cdots \phi_m$.  
  Note that  $\langle \ll t_2\gg, \vec g , \vec \phi,\Psi_2(\vec \phi)\rangle$ is a node in the tree, where $\Psi_2$ is as in item~\eqref{bwist} of Definition \ref{definition12}.
 \item[-] Secondly, put $\xi := \Phi_i(\Psi_2)$ and let $s'$ be the term $\xi t_3[\vec \Psi,\vec g] \cdots t_n[\vec \Phi , \vec g]$. 
 \item[-] Finally, we concatenate with all sequences in $T^*[s';\vec \Phi,\xi]$.
 \end{itemize}
 \item[(vii')] We adjust the construction from the previous item (vi') in analogy with the construction in Definition \ref{definition12}.
 \end{itemize}
 \edefi
We now establish the main result of this section, namely that our alternative computation trees are computable modulo Kleene's $\exists^{2}$.
\begin{theorem}\label{flaheo}
If $t[\vec \Phi,\vec g]$ is a term of type 0 with a value $a$, then $T^\ast[t;\vec \Phi]$ is computable in $\vec \Phi$, $\vec g$, and $\exists^2$.
\end{theorem}
\begin{proof} 
The tree $T^\ast[t;\vec \Phi]$ consists of finite sequences $\tau_1 , \ldots , \tau_n$ of computation tuples. Given a sequence $\tau_1 , \ldots , \tau_n$ of alleged computation tuples, we can check if it belongs to $T^\ast[t;\vec \Phi]$ via recursion on the length as follows.  The procedure for $n = 0$ is trivial.  
If $n>0$ and $\tau_1 , \ldots , \tau_{n-1}$ is accepted, we use $\exists^2$ to check whether $\tau_n^-$, which is $\tau_n$ without the alleged value, codes the right index and the right parameters for the next sub-computation.  Then, using the fact that the computation is terminating, we check if the alleged value is the right one.
\end{proof}

\subsection{Countably based functionals and partiality}\label{kepl}
\subsubsection{Introduction}
The countably based functionals were originally suggested by Stan Wainer and then studied by John Hartley in \cite{hartleycountable,hartjeS}. 
The original class of countably based functionals is a type-hierarchy of hereditarily total functionals. In this section, we extend this notion to partial objects appearing in some $\P(\tau)$ for $\tau \in \Ty$.  
We then establish the following properties of this extended notion.
\begin{itemize}
\item We show in Section \ref{pref} that the partial countably based functionals are closed under computability. 
\item We show that the partial countably based functional $\Omega_{1}$, mentioned in Section \ref{intro} and defined by Definition \ref{defomega1}, is `fundamentally partial' in nature, in that no total functional is equivalent to it (Section \ref{nosim}).
\item We show in Section \ref{cross} that $\Omega_{1}$ and the Suslin functional $\SS$ decide the halting problem for infinite time Turing machines (ITTMs for short), i.e.\ the former combination is thus stronger than ITTMs with type one oracles.
\end{itemize}
As to the naturalness of $\Omega_{1}$, we show in Sections \ref{maink2} that $\Omega_1$ is computationally equivalent to (many) functionals witnessing theorems from mainstream mathematics, thus giving rise to the $\Omega_{1}$-cluster. 

\subsubsection{Closure under computability}\label{pref}
In this section, we introduce the notion of partial countably based functional (Definition \ref{1998}) and show that it is closed under computability (Theorem \ref{denkors}).

\smallskip

First of all, the notion of countably based (total) functional is as follows. 
\begin{definition}[Countably based functional]\label{28}{\em 
Let $\Phi \in \F(\sigma)$ where $\sigma \in \Ty(3)$ and $\sigma = \tau_1 , \ldots , \tau_n \rightarrow 0$, and let $\tau_i = \delta_{i,1} , \ldots , \delta_{i,m_i} \rightarrow 0$ where each $\delta_{i,j}$ is in $\Ty(1)$.  
We say that $\Phi$ is \emph{countably based} if, whenever $\Phi(F_1 , \ldots , F_n) \in \N$, there are countable subsets $X_i \subseteq \F(\delta_{i,1}) \times \cdots \times \F(\delta_{i,m_i})$ such that $\Phi(F_1 , \ldots , F_n) = \Phi(G_1 , \ldots , G_n)$ whenever $F_i$ and $G_i$ are equal on $X_i$ for each $i = 1 , \ldots , n$.
A collection of sets $X_i$ as above is a \emph{support} for the value of the computation.
}\end{definition}
Definition \ref{28} is the classical definition; it is known that all functionals of types in $\Ty(2)$ are countably based and that if $\Phi$ is computable in the countably based \emph{total} functionals  $\Phi_1 , \ldots , \Phi_n$, then $\Phi$ is itself countably based.

\smallskip

Secondly, we extend the definition of countably based functionals to partial ones.
\begin{definition}\label{1998}{\em Let $\Phi \in \Pa(\sigma)$ where $\sigma$ and the $\tau_i$ are as in Definition \ref{28}.  
We say that $\Phi$ is \emph{countably based} if, whenever $\Phi(F_1 , \ldots , F_n) \in \N$, there are countable subsets $X_i \subseteq \F(\delta_{i,1}) \times \cdots \times \F(\delta_{i,m_i})$ such that $\Phi(F_1 , \ldots , F_n) = \Phi(G_1 , \ldots , G_n)$ whenever $\Phi(G_1 , \ldots , G_n) \in \N$ and  $F_i$ and $G_i$ are equal on $X_i$ for each $i = 1 , \ldots , n$.
}\end{definition}
Thirdly, we define the following central functional.  Recall that we identify a set $X \subseteq \N^\N$ with its characteristic function.
\begin{definition}[The $\Omega_{1}$-functional]\label{defomega1}{\em  
We let $\Omega_1(X)$ be defined exactly when $X = \{x\}$ for some $x^{1}$, and then $\Omega_1(X) = x$}
\end{definition}
The functional $\Omega_1$ is countably based because if we know that $X$ has exactly one element, then the only information about $X$ we need in order to find $\Omega_1(X)$ is $X(x)$. 
However, there is no countably based total extension $\Phi$ of $\Omega_1$, since there cannot be a countable support for $\Phi(\emptyset)$.

\smallskip

Finally, partial countably based functionals are closed under computability.
\begin{thm}\label{denkors}
Let $\Phi, \Phi_1 , \ldots , \Phi_n$ be partial functionals of types in $\Ty(3)$ such that $\Phi$ is computable in $\Phi_1 , \ldots , \Phi_n$ and such that all $\Phi_1 , \ldots , \Phi_n$ are countably based. Then $\Phi$ is countably based.
\end{thm}
\begin{proof}
For the sake of notational simplicity, we assume that all functionals are of pure type 3.  
Let $t$ be a term with countably based parameters $\Phi_1, \ldots , \Phi_n$ 
such that for all $F$ such that $\Phi(F) \neq \bot$, we have that $\Phi(F)$ is the value of the tree $T[tF]$.  
By recursion on the  subnodes $s$ in this tree, we will find  countable sets $X_s$ such that for all total $G$, if $T[s[F/G]]$ is well-founded and $F$ and $G$ agrees on $X_s$, then the values of $T[s]$ and $T[s[F/G]]$ are the same. 
The case $s = 0$ is a trivial case, where $X_s = \emptyset$. We now consider four further cases, while for the other cases, we have $X_s = X_{s'}$ where  $s'$ is the one immediate child of $s$ in the tree.
\begin{itemize}
\item The case $s = \case s_1s_2s_3$. Then $X_s$ is the union of $X_{s_1}$ and the relevant $X_{s_i}$ for $i \in \{2,3\}$.
\item The case $s = Fs_1$. We then let $X_s$ be the union of all $X_{s_1a}$ for $a \in \N$ together with the singleton $\{[[s_1]]\}$.
\item The case $s = \Phi_is_1$. Let $Y$ be a countable support for $\Phi_i([[s_1]])$. We then define the following union: $X_s = \bigcup_{f \in Y}X_{s_1f}$.
\item The case $s =( \lambda x s_1)s_2 \ldots s_l$. Then $X_s = X_{s_1[x/s_2]s_3 \ldots s_l}$, which will give the only value we need to protect.
\end{itemize}
That this construction works is trivial by induction on the rank of $s$ in $T[tF]$. \end{proof}

\subsubsection{A no-simulation result}\label{nosim}
We show that total functionals cannot simulate the partial functional $\Omega_{1}$ from Definition \ref{defomega1}.

\begin{theorem}\label{31}
The functional $\Omega_1$ is not computable in any \textbf{total} countably based functional of rank 3.
\end{theorem}
\begin{proof}
For the sake of notational simplicity, we again assume that $\Phi$ is of pure type 3.
Assume that $\Omega_1$ is computable in $\Phi$ where $\Phi$ is total and countably based, i.e.\ there is a term $t$ with $\Phi$ as a parameter such that $\Omega_1(F)(k)  \preceq [[tFk]]$ for all $F$ and $k$. 
We shall obtain a contradiction by considering what happens to this inequality for $F = 0^2$, the constant zero functional of type 2, and $k = 0$ as follows.  
\begin{itemize}
\item The case $[[t0^20]] \in \N$ yields a contradiction by considering the countable support $X$ for this value, i.e.\ a countable set $X$ such that if $F$ is constant zero on $X$ and $\Omega_1(F) \in \N^\N$ then $\Omega_1(F)(0) = [[t0^20]]$. 
Clearly, there is $F$ such that $F(x)=0$ for $x\in X$, $\Omega_1(F)$ is defined, and $\Omega_1(F)(0)\neq [[t0^20]]$.
\item The case $[[t0^20]] = \bot$ similarly yields a contradiction.  Indeed, the tree $T[t0^20]$ cannot have a value, but since all inputs and parameters are total no branch ends in $\bot$.  
Hence, the tree must be ill-founded and we choose an infinite branch in $T[t0^20]$ in such a way that there is a countable support $X$ for the branch being infinite. 
We do this by a top-down recursion following how the tree is defined.  Hence, assume that item no.\ $k$ in the branch is $t_k$, starting with $t_0 = t0^20$, and assume that the tree below $t_k$ is ill-founded. 
Sub-terms $s$ in the tree may have $\Phi$ and $0^2$ as parameters, but no other parameters with rank 2 or 3.  Hence, if $T[s]$ has a value, there is a countable support for this value.  We proceed via the following case-distinction. 
\begin{itemize}
\item The case $t_k = 0$ is impossible; for the cases $t_k = \suc t_{k+1}$, $t_k = \pd t_{k+1}$ and $t_k = (\mu x s)s_1 \cdots s_l$, there is only one option for the next item.
\item The case $t_k = \case s_1 s_2 s_3$. If $T[s_1]$ is ill-founded, we let $t_{k+1} = s_1$. If $T[s_1]$ is well-founded, and then with a value, we let $t_{k+1}$ be $s_2$ or $s_3$ according to the value. We then add a support for the value of $s_1$ to the support we are building up.

\item The case $t_k = 0^2s$. Then there must be an integer $a$ such that $T[sa]$ does not have a value, and we let $t_{k+1}$ be one such $sa$.
\item The case $t_k = \Phi s$. Then there is an $f \in \N^\N$ and an $a \in \N$ such that $T[sfa]$ does not have a value. Let $t_{k+1}$ be one such $sfa$.
 This is where we use the assumption that $\Phi$ is total, the existence of $f$ and $a$ as above actually depends on it.
\item The case $t_k = (\lambda x s_1)s_2 \cdots s_l$. There are two possibilities, that $[[s_2]]$ is not total and that $[[s_2]]$ is total while $T[s_1[x/s_2]s_3 \cdots s_l]$ does not have a value. In the first case we proceed as above, while in the second case we let $t_{k+1} = s_1[x/s_2]s_3 \cdots s_l$. Since $s_2$ does not have to be of pure type, we must let $t_{k+1} = s_2 \vec g$ for some list $\vec g$ of total arguments in the first case.
\end{itemize}
\end{itemize}
We only add a finite sequence to the support in the case of $\case$. The final support $X$ is therefore countable. 
Next, we prove the following claim.  
\begin{claim}\label{xkx}
If $F^{2}$ is constant 0 on $X$, then $T[tF0]$ is without a value. 
\end{claim}
\begin{proof}
If $s$ is a term with $0^2$ as a parameter, we let $s^* =  s[0^2/F]$ throughout.
If $t_0^*, t_1^* , t_2 ^*, \ldots$ is an infinite branch in $T[tF0]$, this tree is not well-founded, and cannot have a value. The other possibility is that there is $k$ such that $(t_0^* , \ldots , t_k^*) \in T[tF0]$ but $t_0 ^*, \ldots , t_{k+1}^*$ is not. There will be only two situations where this possibility can occur, which we study now. 
\begin{itemize}
\item The case $t_k = \case s_1 s_2 s_3$ where we have chosen $t_{k+1} = s_i$ for $i = 2$ or $i = 3$. 
If $T[s_1^*]$ is not well-founded, the tree $T[tF0]$ is not well-founded, and does not have a value. If the tree $T[s_1^*]$ has  a value it must, due to our choice of support, have the same value as $T[s_1]$.  Hence, our choice of $t_{k+1}$ transfers to $t^*_{k+1}$, contradicting the choice of $k$.
\item The case $t_k = (\lambda x s_1)s_2 \cdots s_l$. If we have chosen $t_{k+1} = s_2\vec g$, then $t^*_{k+1}$ is an extension of $t^*_k$ in $T[tF0]$, so, by the choice of $k$ we must have that $t_{k+1} = s_1[x/s_2]s_3 \cdots s_l$,  If $[[s_2^*]]_w$ is total, then $t^*_{k+1}$ would be a correct continuation of $t^*_k$ in $T[tF0]$, so we must have that $[[s_2^*]]_w$ is not total, and the tree $T[tF0]$ will not have a value.
\end{itemize}
The proof of Claim \ref{xkx} is now finished in light of the previous case distinction.  
\end{proof}
Finally, together with the above case distinction, Claim \ref{xkx} leads to a contradiction for any $F$ that is constant 0 on $X$, takes the value 1 at exactly one point outside $X$, and is zero elsewhere. 
Hence, the proof of Theorem \ref{31} is finished.
\end{proof}
Finally, Corollary \ref{kink} explains why the notion of countably based is useful: similar to Grilliot's trick (see \cite{kohlenbach2}), the former yields a `rough-and-ready' classification.  
\begin{corollary}\label{kink}
Let $\Psi$ be total and countably based and let $\Xi$ be total and of type rank $\leq 3$.  The pairs $\Psi,\Omega_1$ and $\Psi,\Xi$ are not computationally equivalent.
\end{corollary}
\begin{proof}
If $\Xi$ is computable in $\Psi+\Omega_1$, the former is countably based. However, then $\Omega_1$ is not computable in $\Xi$ and $\Psi$ by Theorem \ref{31}.
\end{proof}

\subsubsection{An explosive example}\label{cross}
In this section, we show that the combination of $\Omega_{1}$ and the Suslin functional $\SS^{2}$ is rather powerful in that it computes the Halting problem for ITTMs.
This is mainly a consequence of the Kondo-Addison uniformisation theorem for $\Pi^1_1$-sets, to be found in \cite{Rogers}*{Ch.\ 16, Thm.\ XLV}.
We note that $\Omega_1$ is rather weak by itself, at least when it comes to computing objects of type 1, even in combination with $\exists^2$ (see \cite{dagsamXI}*{\S4}).  

\smallskip

We first need the following lemma.
\begin{lemma}\label{lemmaKA} 
Let $A \subset \N^\N$ be a non-empty $\Pi^1_1$-set. Then $A$ has an element computable in $\Omega_1$ and $\SS$.
\end{lemma}
\begin{proof}
By uniformisation, there is a $\Pi^1_1$-set $B \subseteq A$ with exactly one element $x$, and since $B$ is computable in $\SS$ we have that $x$ is computable in $\Omega_1$ and $\SS$.
\end{proof}
\begin{theorem}\label{jable}~
\begin{itemize}
\item[(a)] If $A \subset \N$ is computable by an ITTM, then $A$ is computable from $\Omega_1$ and $\SS$.
\item[(b)] The halting problem for ITTMs is computable in $\Omega_1$ and $\SS$.
\end{itemize}
\end{theorem}
\begin{proof}
If $x \in \N^\N$ codes a well-ordering of length $\alpha$ and $M$ is an ITTM with number-code $m$, then we can simulate the calculation of $M$ through $\alpha$ steps computably in $x$ and $\exists^2$. There is an ordinal $\alpha$ such that for all ITTMs $M$ and all integer inputs $a$, $M$ with input $a$ has either come to a halt at stage $\alpha$ or been through sufficiently many loops to ensure that it will run forever; see \cite{welch2} for details. 
The set of codes or ordinals $\alpha$ with this property is a $\Pi^1_1$-set and we can apply Lemma \ref{lemmaKA}.
\end{proof}

\subsection{Approximations and partiality}\label{kepl2}
\subsubsection{Introduction}
We have shown in Section \ref{nosim} that $\Omega_{1}$ is `fundamentally partial' in nature, in that no total functional (of rank $3$) is computationally equivalent to it. 
In this section, we establish the same result for $\Omega$ mentioned in Section \ref{intro}.
\bdefi[The functional $\Omega$]\label{defomega} 
The functional $\Omega(X)$ is defined if $|X|\leq 1$ and outputs a finite sequence that includes all elements of $X$.
\edefi
By Theorem \ref{fct}, the functional $\Omega$ can enumerate \emph{any} finite set, which is why we use the term `finiteness' functional in Section \ref{intro}.  
In this section, we obtain the following fundamental results pertaining to $\Omega$. 
\begin{itemize}
\item We show that $\Omega$ is \emph{weak} when combined with $\exists^{2}$ (Section \ref{tris}).  To this end, we introduce another functional $\Omega_{\textup b}$ (Definition \ref{omegab}).
\item We obtain a rather general procedure for approximating computations in the sense of our new model   (Section \ref{sec2}).
\item We show that there is no total extension of the functional $\Omega_{\textup b}$ that is computable from $\Omega_{\textup b}$ and $\exists^{2}$.    (Section \ref{sec33}).
\item Based on the previous three items, we show that there is no total functional (of rank 3) that is computationally equivalent to $\Omega$.    (Section \ref{sec4}).
\end{itemize}
In conclusion, we may claim that $\Omega$ is (also) fundamentally partial in nature.

\subsubsection{The strength of finiteness functionals}\label{tris}
In this section, we show that $\Omega$ is \emph{weak}, even when combined with $\exists^{2}$, as in Theorem \ref{weakk}.
By contrast, the $\Sigma^1_2$-comprehension functional $\SS^2_2$ is computable in $\Omega$ and the Suslin functional $\SS^2_1$ (\cite{dagsamXI}*{\S4}). 
Moreover, assuming \textsf{V = L}, $\exists^3$ is computable in $\Omega$ and $\SS^2_1$. 
Since there is a $\Delta^1_2$-well-ordering of the continuum if \textsf{V = L}, this is also a consequence of Theorem~\ref{thmwo}.

\smallskip

Moreover, we have shown that $\Omega$ is close to $\exists^3$ in the sense that the latter is computable in $\Omega$ and some functional of type 2, namely a well-ordering of $[0,1]$; 
see Theorem~\ref{thmwo} or \cite{dagsamXII} for this result. 
On the other hand, left with extra parameters of type 1 only, $\Omega$ supplies no extra computational power to $\exists^2$ by Theorem \ref{weakk}. 

\smallskip

In order to obtain the aforementioned results in a transparent way, we introduce another functional, namely $\Omega_{\rm b}$, where the sub-script `b' stands for \emph{basic}.
\bdefi[The functional $\Omega_{\textup b}$]\label{omegab} 
Let $\Omega_{\rm b}(X)$ be defined whenever $X \subset \N^\N$ has at most one element, and then $\Omega_{\rm b}(X) = 1$ if $X$ is nonempty and $\Omega_{\rm b}(\emptyset) = 0$.
\edefi
\noindent
Despite its basic nature, the functional $\Omega_{\textup b}$ computes the functional $\Omega$.
\begin{lemma}\label{borkim2}
The functional $\Omega_{\rm b}$ is computationally equivalent to $\Omega$; the latter is not countably based.
\end{lemma}
\begin{proof} The functional $\Omega_{\rm b}$ is clearly computable in $\Omega$. For the converse, if $\Omega_{\rm b}(X) = 0$ then $\Omega(X) = 0^\omega$ (or any  other specified value coding the empty sequence). If $\Omega_{\rm b}(X) = 1$ and $f$ is the single element in $X$ we can compute $f(n)$ from $\Omega_{\rm b}$ and $X$ by considering, for each $k$, $\Omega_{\rm b}(\{g \in X : g(n) = k\})$. 
For the second part of the lemma, it suffices to show that $\Omega_{\rm b}$ is not countably based. This is obvious, since there cannot be a countable support for $\Omega_{\rm b}(\emptyset)$.
\end{proof}
The functional $\Omega_{\rm b}$ is useful, because it is so simple. Indeed, when we prove that $\Omega$ is computable in a given object, it suffices to do so for $\Omega_{\rm b}$, and the argument is then often quite easy. Here, we show that $\Omega_{\rm b}$ has no computational power by itself.
\begin{theorem} \label{weakk}
For each $f \in \N^\N$ and  finite sequence $\vec g$ from $\N^\N$, if $f$ is computable in $\Omega_{\rm b}+\exists^2+\vec g$, then $f$ is computable in $\exists^2+\vec g$.
\end{theorem}
\begin{proof}
If a set $X$ is hyperarithmetical in $\vec g$ and has at most one element, then 
\be\label{drif}
X \neq \emptyset \leftrightarrow \exists x(x \in X )\leftrightarrow \big(\exists y \in \textsf{HYP}(\vec g)\big)(y \in X).
\ee
In light of \eqref{drif}, $\Omega_{\rm b}$ restricted to sets computable from $\exists^2$ and some parameters of type 1 is actually computable in $\exists^2$ and the same parameters. We combine this with recursion over computation trees to eliminate the use of $\Omega_{\rm b}$ as long as all extra parameters are of type 1. 
\end{proof}

\subsubsection{An approximation theorem}  \label{sec2}
In this section, we obtain a rather general approximation result, namely Theorem \ref{thm.mol}, for our equivalent $\lambda$-calculus model of S1-S9 computability.  
The procedure of approximating a partial computable function is well-known in certain cases, as discussed first in Remark \ref{partapp}. 
\begin{rem}[Approximating computations]\label{partapp}\rm
First of all, a most basic example of approximating a partial computable function by total ones is the $n$-th approximation to a Turing machine (TM) calculation.  
Here, we output 0 or some other fixed value if the TM does not come to a halt within $n$ steps, in analogy with defining the primitive recursive function $\varphi_{e,n}$ using Kleene's $T$-predicate.  

\smallskip

Secondly, assuming all inputs are total, one defines the $n$-th approximation $\{e\}_n(\vec \Phi)$ to a Kleene-computation by a top-down evaluation: simply replace $n$ by $n-1$ for sub-computations and only use the correct values for all $n$ for computations of rank 0. Then $\{e\}_n(\vec \Phi)$ will be total and defined by recursion on $n$. If all inputs are \emph{continuous}, we can iterate Grilliot's trick to show that we get a correct approximation to all terminating computations as $n$ increases.
\end{rem}
Below, we obtain an approximation result for \emph{arbitrary} total inputs based on ordinals; we show that as long as the associated well-founded computation trees have rank bounded by an ordinal $\alpha$, the value of the $\alpha$-approximation is the true value.  In our construction, we make use of the following notations.
\begin{defi}\label{gank}\rm
 If $\alpha$ is an ordinal and $n_\beta \in \N$ for all $\beta < \alpha$, we define $\lim^*_{\beta \rightarrow \alpha}n_\beta$ as $n$ if there is some $\beta < \alpha$ such that $n = n_\gamma$ for all $\gamma$ with $\beta \leq \gamma < \alpha$; this limit is 0 if there is no such value $n$. 
\end{defi}
Note that if $\alpha = \beta + 1$, the limit from Definition \ref{gank} is $n_\beta$ by definition.
\begin{nota}\rm
Let $t$ be a term of type 0 with parameters $\vec \Phi$, $\vec F$, $\vec f,\vec a$  of pure types only. If $[[t(\vec \Phi,\vec F , \vec f , \vec a)]] \in \N$, there is a well-founded computation tree supporting this fact by Definition \ref{definition12}; we let $\rho(t(\vec \Phi,\vec F, \vec  f, \vec a)$ be the ordinal rank of this tree if the value is in $\N$, and $\infty$  otherwise. 
By abuse of notation, we also use  $t(\vec \Phi, \vec F , \vec f , \vec a)$ for the interpretation 
$[[t(\vec \Phi,\vec F , \vec f , \vec a)]] $ when there is no ambiguity.
\end{nota}
The following definition follows the cases of Definition \ref{definition12}.
\begin{definition}[Approximating terms]\label{def11}{\em 
Let $\alpha$ be an ordinal and let $t(\vec \Phi, \vec F , \vec f , \vec a)$ be a term with all parameters total and among $\vec \Phi, \vec F , \vec f , \vec a$.  
By recursion on $\alpha$, we define $t_\alpha(\vec \Phi, \vec F , \vec f , \vec a)$ as 0 if $t$ is the constant 0 (item~(ii) in Def.\ \ref{definition12}) or if $\alpha = 0$ (item~(i) in Def.\ \ref{definition12}).  If there are one or two  immediate sub-terms $s(\vec \Psi , \vec G , \vec g , \vec b)$ in the computation tree, we use the values $\lim^*_{\beta \rightarrow \alpha}s_\beta(\vec \Psi , \vec G , \vec g , \vec b)$ instead of the true  values when defining the value of $t_\alpha(\vec \Phi, \vec F , \vec f , \vec a)$ (items (iii)-(v) and (viii) in Def.\ \ref{definition12}). We treat the two remaining cases in more detail, as follows.
\begin{itemize}
\item[(vi)] For $t = \Phi_i t_2 \cdots t_n$, the term $t_2$ is of type $\tau = (\delta_1 , \ldots , \delta_m \rightarrow 0)$. 
 \begin{itemize}
  \item For $(\phi_1 , \ldots , \phi_m) \in \F(\delta_1) \times \cdots \times \F(\delta_m)$, calculate $\lim^*_{\beta \rightarrow \alpha} s_\beta$,  where $s = t_2\phi_1 \cdots \phi_m$.  Then define $(\Psi_2)_\alpha(\phi_1 , \ldots , \phi_m)$ as this limit.
 \item Now define $\xi_\alpha: = \Phi_i((\Psi_2)_\alpha)$,  which is where we use that $\Phi_i$ is total
 Define $s'$ as the term $(\xi )_\alpha t_3 \cdots t_n$, as in item \eqref{bwist} of Def.\ \ref{definition12}.
 \item Finally, we define the approximating term $t_\alpha :=\lim^*_{\beta \rightarrow \alpha}(s')_\beta$ 
 \end{itemize}
 \item[(vii)] For $t =( \lambda x t_1)t_2 \cdots t_n$, we  consider the term $s' = t_1[x/t_2]t_3 \cdots t_n$ and define $t_\alpha = \lim^*_{\beta \rightarrow \alpha}(s')_\beta$.
\end{itemize}
}\end{definition}
We have the following theorem regarding approximating terms. 
\begin{theorem}\label{thm.mol}
Let $t$ be a term as above, i.e.\ all parameters are total and among $\vec \Phi$, $\vec F$, $\vec f,\vec a$. 
For any ordinal $\alpha$, the approximating term $t_{\alpha}$ is in $\N$ and
\begin{itemize} 
\item if $\rho(t(\vec \Phi , \vec F , \vec  f, \vec a) )\leq \alpha$, then  $t_\alpha(\vec \Phi, \vec F , \vec f , \vec a) = t(\vec \Phi, \vec F , \vec f , \vec a)$.
\end{itemize}
For $x$ a code of a well-ordering of $\N$ and a countable ordinal $\beta$, we have that:
\begin{itemize}
\item we can compute $t_\beta(\vec \Phi, \vec F , \vec f , \vec a)$, uniformly in $x$, $\exists^{2}$,  and the other parameters.
\end{itemize}
\end{theorem}
\begin{proof}
Firstly, that $t_{\alpha}\in \N$ is immediate by induction on $\alpha$, while for the final item we observe that $\lim^*_{\beta \rightarrow \alpha}n_\beta$ is computable uniformly in $x$, $\exists^2$, and (a coding of) a sequence $\{n_\beta\}_{\beta < \alpha}$. 

\smallskip
  
For the remaining item, we may apply the induction hypothesis to all immediate sub-terms $s$ in the computation tree.   Then $\lim^*_{\beta \rightarrow \alpha}s_\beta$ provides the right value of all sub-terms as in Def.\ \ref{definition12}. 
If we are in item (vi) in Def.\ \ref{def11}, then $(\Psi_2)_\alpha = \Psi_2$ and consequently $\xi_\alpha = \xi$, as defined in the constructions. 
If we are in item (vii) in Def.~\ref{def11}, we do not need to consider modified $\beta$-conversion, because plain $\beta$-conversion will, under our assumptions, provide the right value.
\end{proof}

\subsubsection{Extensions of the functional $\Omega_{\rm b}$}\label{sec33}
In this section, we show that there is no total extension of $\Omega_{\rm b}$ that is computable in $\Omega_{\rm b}$ and $\exists^2$.
A crucial result to this end is Theorem \ref{weakk} about the weakness of $\Omega_{\rm b}$ in the presence of $\exists^{2}$.

\smallskip

First of all, let $0^2$ be the constant zero functional of type 2, also seen as the characteristic function of the empty set $\emptyset$.
\begin{lemma}\label{lemmaO} Let $t(\Omega_{\rm b} , 0^2 ,\exists^2, \vec f)$ be a term with a value, where each element of $\vec f$ either is a hyperarithmetical function or just an integer. Let $F$ be a total functional of pure type 2, computable in $\exists^2$, such that $F(g) = 0$ whenever $g$ is hyperarithmetical. 
If $t(\Omega_{\rm b} , F,{\exists^{2}},\vec f)$ has a value, then the value is that of $t(\Omega_{\rm b},0^2,\exists^{2},\vec f)$. \end{lemma}
\begin{proof}
The proof proceeds by induction on the rank of the computation tree for $t(\Omega_{\rm b},0^2,\exists^2,\vec f)$.  The induction step is trivial except for two cases, namely application of $0^2/F$ and application of $\Omega_{\rm b}$, which we consider as follows.
\begin{enumerate}
\item \textbf{Application of $0^2/F$.} Suppose $t(\Omega_{\rm b},0^2,\exists^2,\vec f) = 0^2(g)$ where we have $g(a) = s(\Omega_{\rm b}, 0^2 , \exists^2,\vec f , a)$. Then $g$ is hyperarithmetical and by the induction hypothesis, $g(a) = s(\Omega_{\rm b},F,\exists^2,\vec f  , a)$ for each $a$. Since, $F(g) = 0$ by assumption, the induction step follows.

\item \textbf{Application of $\Omega_{\rm b}$.} Suppose $t(\Omega_{\rm b},0^2,\exists^2,\vec f) = \Omega_{\rm b}(G_0)$, where we have $G_0(g) = s(\Omega_{\rm b} , 0^2 , \exists^2, \vec f , g)$, and define $t(\Omega_{\rm b},F,\exists^2,\vec f) = \Omega_{\rm b}(G_F)$ in the same way. Since both terms are assumed to have values, $G_0$ and $G_F$ are characteristic functions of sets that either are the empty set or singletons. 
Note that in the latter case, the elements of the singletons must be hyperarithmetical.  Indeed, both $G_0$ and $G_F$ are computable in $\exists^2$ by (the proof of) Theorem \ref{weakk}, and thus hyperarithmetical.
We use these observations and the induction hypothesis to show that $G_{0}=G_{F}$. If $G_0(g) = 1$ for some $g$, then, by the induction hypothesis, $G_{F}(g) = 1$, and since both functions are characteristic functions of singletons, they must be equal. The argument for the other direction is the same.
\end{enumerate}
Having established the two remaining cases, we are done.
\end{proof}
Theorem \ref{thm.totext} is interesting in its own right, but also essential for Section~\ref{sec4}.
\begin{theorem}\label{thm.totext} 
There is no total extension of $\Omega_{\rm b}$ that is computable in $\Omega_{\rm b}+\exists^2$.
\end{theorem}
\begin{proof}
Assume that $\Phi$ is total, extends $\Omega_{\rm b}$, and is computable in $\Omega_{\rm b}+\exists^2$. 
By Lemma~\ref{lemmaO}, $\Phi(X) = 0$ if $X$ is hyperarithmetical, but without hyperarithmetical elements.
We can now use the following concepts and facts from the (original) proof of the Gandy-Spector theorem (see \cite{Gandy1960, Spector1959} and also \cite{Sacks.high}*{Section III.3}).
\begin{itemize}
\item If $( A , < )$ is a Turing-computable ordering, a corresponding \emph{jump-chain} is a sequence $\{C_a\}_{a \in A}$ such that or all $a \in A$, $C_a$ is the Turing jump of $\{\langle b , c\rangle : b < a \wedge c \in C_b\}$.
\item If $(A , < )$ is a well-ordering, there is a unique jump-chain, which is also hyperarithmetical. 
\end{itemize}
Recall that the set of computable well-orderings is complete $\Pi^1_1$, while the set of computable orderings with a jump chain and without hyperarithmetical infinite descending sequences,  is $\Sigma^1_1$.  
Hence, there is a computable ordering $(A,<)$ of the latter kind that is not well-ordered. Let $A_w$ be the well-ordered initial segment of $A$ and observe the folllowing.
\begin{enumerate}
\renewcommand{\theenumi}{\alph{enumi}}
\item If $a \in A$, then the set $X_a$ of jump-chains up to $a$ is arithmetically defined, uniformly in $a$, and non-empty.\label{hakker}
\item If $a \in A_w$ there is exactly one jump-chain up to $a$.
\item If $a \in A \setminus A_w$, then there is no hyperarithmetical jump-chain up to $a$.\label{hakker4}
\item The initial segment $A_w$ is complete $\Pi^1_1$.\label{hakker5}
\end{enumerate} 
By Lemma \ref{lemmaO} and items \eqref{hakker}-\eqref{hakker4}, we have the following equality
\be\label{plap}
A_w = \{a \in A : \Phi(X_a) = 1\},
\ee
implying that $A_w$ is computable in $\Phi$. By item \eqref{hakker5}, the set $A_w$ from \eqref{plap} is a complete $\Pi^1_1$-set, contradicting that all functions in $\N^\N$ computable in $\Phi$ (that is computable in $\Omega_{\rm b}$ and $\exists^2$) must be hyperarithmetical.
This ends the proof.
\end{proof}

\subsubsection{Extensions of the functional $\Omega$}\label{sec4}
In this section, we show that there is no total functional of type 3 that is computationally equivalent to $\Omega$.  We make essential use of the previous sections and a well-known theorem of descriptive set theory.

\smallskip

First of all, we establish some essential lemmas, where we now use the alternative Definition \ref{atrees} of computation trees.
\begin{lemma}\label{lemma4} 
Let $\Phi$ be a total functional of type 3 computable in $\Omega_{\rm b}$. Let $F:(\N^\N)^2 \rightarrow \N$ be computable in $\exists^2$, and let $t$  be a term such that $t(\Phi ,F _f,\exists^2)$ has a value for each $ \in \N^\N$, where $F_f(g) = (f,g)$ for each $f$ and $g$. Then the map sending $f$ to the well-founded computation tree of $t(\Phi,F_f, \exists^2)$ is computable in $\exists^2$. 
\end{lemma}
\begin{proof}
This is follows from the fact that the computation tree of $t(\Phi,F_f,\exists^2)$ is computable in $\Phi$, $f\in \N^{\N}$, and  $\exists^2$, plus the assumption that $\Omega_{\rm b}$ computes $\Phi$.
\end{proof}
\begin{lemma}\label{lemma5}
Let $\Phi$, $F$, and $t$ be as in Lemma \ref{lemma4}. There is a Turing-computable well-ordering $\prec$ of $\N$ such that the ordinal rank of $\prec$ exceeds the rank of the computation trees of $t(\Phi , F_f , \exists^2)$ for all $f\in \N^{\N}$.
\end{lemma}
\begin{proof}
The fact that these well-founded computation trees are hyperarithmetical suffices. For the unfamiliar reader, we observe that the set of trees on $\N$ that can be order-embedded into any of these computation trees, is a $\Sigma^1_1$-set of well-founded trees on $\N$; it is well-known that the ranks of the elements of such a set are (uniformly) bounded below $\omega_1^{\textsf{CK}}$.
\end{proof}
Finally, we can prove the main theorem of this section.  
\begin{theorem} 
No total $\Phi^{3}$ is computationally equivalent to $\Omega$, assuming $\exists^{2}$.
\end{theorem}
\begin{proof}
Assume there is such a $\Phi$ and let $t$ be a term such that $\lambda G. t(\Phi, G , \exists^2)$ is a (possibly partial) extension of $\Omega_{\rm b}$.
Now define $F \leq 1^2$ as follows:
\[
F(f,g) = 1 \leftrightarrow (\forall a \in \N)(f(a) > 0  \wedge g(a) = f(a)-1).
\] 
All arguments for which $\Omega_{\rm b}$ is defined, are of the form $F_f$.
Following Lemma~\ref{lemma5}, let $x$ be a Turing-computable code of a well-ordering with a rank $\alpha$ that is larger than the rank of all computations   trees for $t(\Phi , F_f , \exists^2)$ for any $f \in \N^\N$. 
By Theorem~\ref{thm.mol}, the term $\lambda G.t_\alpha(\Phi , G, \exists^2)$ is total and computable.  By the choice of $F$ and $x$, we also have that this total functional extends $\Omega_{\rm b}$. 
By Theorem~\ref{thm.totext} this is impossible. 
\end{proof}
In conclusion, we observe that $\Omega$ is fundamentally partial in nature.

\section{Two robust computational clusters}\label{maink2}
\subsection{Introduction}
As discussed in Section \ref{intro}, many functionals stemming from theorems in mainstream mathematics are computationally equivalent to either the $\Omega$ or $\Omega_{1}$.  
We refer to the associated computational equivalence classes as the \emph{$\Omega$-cluster} and the \emph{$\Omega_{1}$-cluster}.   
In this section, we populate these clusters with functionals stemming from the following topics.
\begin{itemize}
\item Finiteness and countability of subsets of $\R$ (Section \ref{struc}).
\item Regulated functions on the unit interval (Section \ref{sams}).
\item Functions of bounded variation on the unit interval (Sections \ref{sams}-\ref{separ}).
\item Absolutely continuous functions on the unit interval (Section \ref{maink3}).
\item Functions in the Sobolev space $W^{1,1}$ (Section \ref{flukkkk})
\item \emph{Caccioppoli} or \emph{finite perimeter} sets, which are essentially sets with characteristic function of bounded variation  (Section \ref{caccio})
\end{itemize}
On a conceptual note, $\Omega_1$ and $\Omega$ are examples of what we call \emph{a structure functional}, i.e.\ a functional that does not turn up as a result of some construction in mainstream mathematics, but which is nonetheless useful for calibrating the computational complexity of those that do. 

\subsection{Functionals related to finiteness and countability}\label{struc}
In this section, we connect the $\Omega$-functional to other functionals performing basic operations on finite or countable sets of reals, as in the following definition.  
\bdefi[Functionals witnessing properties of finite sets]\label{JDR0}
Let $X\subset [0,1]$ be a set of reals and let $n\in \N$.  
\begin{itemize}
\item A \emph{finiteness realiser} $\Omega_{\fin}$ is defined when the input $X$ is finite and outputs a finite sequence that includes all elements of $X$.
\item The functional $\Omega_n$ is defined if $|X|=n$ and outputs a finite sequence that includes all elements of $X$.  
\item  The functional $\Omega_{\leq n}$ is defined if $|X|\leq n$ and outputs a finite sequence that includes all elements of $X$.
\item The functional $\Omega_{\#,\fin}$ is defined when $X$ is finite and outputs $|X|$. 
\item The functional $\Omega_{\geq \#,\fin}$ is defined when $X$ is finite and outputs some $n\geq |X|$. 
\end{itemize}
\edefi
We note that $\Omega$ is actually $\Omega_{\leq 1}$.
While perhaps not clear at first glance, the previous `finiteness' functionals are intimately related to functionals witnessing basic properties of functions of bounded variation, as studied in Section \ref{sams}
\bdefi[Clusters] 
The \emph{$\Omega$-cluster} is the class of partial functionals computationally equivalent to $\Omega$ modulo $\exists^2$.  The \emph{$\Omega_1$-cluster} is the class of partial functionals computationally equivalent to $\Omega_1$ modulo $\exists^2$.
\edefi
When we define a functional via an incomplete description, we are actually defining a \emph{class} of functionals, see e.g. $\Omega_{\geq \#,\fin}$ above. A class like this is in the $\Omega$-cluster if $\Omega$ is computable in all elements of the class modulo $\exists^2$ and at least one element of the class is in the $\Omega$-cluster. This extends to $\Omega_1$ in the same way.

\smallskip

Since $\exists^2$ is always assumed, we do not have to be specific about which of the domains $\N^\N$ , $2^\N$, $[0,1]$ or $\R$ we are working with when we define e.g.\ $\Omega$ and $\Omega_1$.  
\bdefi[Functionals related to countability]\label{JDR4}
\begin{itemize}
\item An \emph{enumeration functional} is a partial functional taking as input $A\subset [0,1]$ and $Y:[0,1]\di \N$ injective on $A$, outputting an enumeration of $A$. 
\item An \emph{$\Omega_{\BW}^{3}$-functional} is a partial functional taking as input $A\subset [0,1]$ and $Y:[0,1]\di \N$ injective on $A$, and outputting $\sup A$.
\item A \emph{\textbf{weak} enumeration functional} is an enumeration functional that only works if the second input is also surjective on the first. 
\item A \emph{\textbf{weak} $\Omega_{\BW}$-functional} is an $\Omega_{\BW}$ functional that only works if the second input is also surjective on the first. 
\end{itemize}
\edefi
\noindent
We note that `$\BW$' in the previous definition stands for `Bolzano-Weierstrass'. We have the following initial classification.  
\begin{theorem}[First cluster theorem]\label{fct}~
\begin{itemize}
\item The $\Omega_1$-cluster contains each $\Omega_n$ for $n \geq 1$, a weak enumeration functional, and the weak $\Omega_{\BW}$  functional.
\item The $\Omega$-cluster contains $\Omega_{\fin}$, $\Omega_{\leq n}$, $\Omega_{\#,\fin}$, $\Omega_{\geq \#,\fin}$, an enumeration functional, and $\Omega_{\BW}$.
\end{itemize}
\end{theorem} 
\begin{proof} 
Some of the results in the theorem are easy and are left for the reader.
In particular, it is straightforward to show that $\Omega_1$ and $\Omega$ are computable in most functionals mentioned in the theorem.  
We do establish that  $\Omega_{\fin}$ is computable in $\Omega$, based on the fact that all domains in question have natural arithmetically defined linear orderings. Note that $\Omega$ can decide if a set $X$ with at most one element is empty or not, simply by asking if $\Omega(X) \in X$ or not.  We use self-reference, i.e.\ the $(\mu xt)$-facility present in our $\lambda$-calculus approach to S1-S9, to define a functional $\Omega^\ast$ with the following properties for finite sets $X$:
\begin{itemize} 
\item $\Omega^\ast(\emptyset) $ is the special element $0^\omega$,
\item $\Omega^\ast(X)$ is the least element in $X$ when $X$ is non-empty.
\end{itemize}
Without loss of generality, we may assume that $\Omega(\emptyset) = 0^\omega$ is the least element in our domain. We then use the following algorithm: 
\begin{center}
given $X$, put $x \in Y$ if $x \in X$ and $\Omega^\ast(\{y \in X : y < x\}) \not \in X$ and define $\Omega^\ast(X)  := \Omega(Y)$. 
\end{center}
We see, by induction on the cardinality of the finite set $X$, that $\Omega^\ast(X)$ is well-defined and that it selects the least element of $X\ne \emptyset$. 
We can then use $\Omega^\ast$ to enumerate $X$ from below until there is nothing left, thus obtaining $\Omega_{\fin}$ as follows.

\smallskip

Finally, we show that $\Omega_{\geq \#, \fin}$ is part of the $\Omega$-cluster. 
To this end, let $X\subseteq [0,1]$ have at most one element and suppose $\Omega_{\geq\#, \fin}(X) = k$.
Let the closed intervals $ I_0 , ... , I_{k}$ form a partition of $[0,1]$ and let $f_i:I_i \di [0,1]$ be the associated canonical (affine) bijection.
Now define a set $Y\subset [0,1]$ as follows: $y \in Y$ if for some $i \leq k$ we have that $y \in I_i$ and $f_i(y) \in X$.
Then we have the following: 
\be\label{exti}
(\exists x \in [0,1])(x\in X)\asa \Omega_{\geq \#, \fin}(Y)> k,
\ee
which readily yields the functional $\Omega$.  Indeed, in case $X$ is not empty, we can repeat the previous for $X\cap [q,r]$ and $X\setminus [q,r]$ for any $q, r\in \Q\cap [0,1]$.
In this way, the usual interval-halving technique allows us to find the single element of $X$.
\end{proof}
We note that \eqref{exti} relies on the axiom of extensionality for functions (of relatively high type).
The idea behind the functional $\Omega^\ast$ from the previous proof can also be exploited as follows. 
\begin{theorem}\label{thmwo} 
There is a functional $\Omega_{\rm WO}$ in the $\Omega$-cluster such that $\Omega_{\rm WO}(X,Y,\prec)$ is defined whenever $\prec$ is a well-ordering of $X$ and $Y \subseteq X$, and if $Y \neq \emptyset$ then $\Omega_{\rm WO}(X,Y,\prec)$ is the $\prec$-least element of $Y$. 
\end{theorem}
\begin{proof} Without loss of generality we may assume that $\Omega(\emptyset) \not \in Y$. We then define $\Omega_{\rm WO}$ via the following self-referential program:
\begin{center}
for $X$, $Y,\prec$, we define the set $Z$ by $y \in Z$ if $y \in Y$ and $\Omega_{\rm WO}(X, \{z \in Y : z \prec x\},\prec) \not \in Y$ and put $\Omega_{\rm WO}(X,Y,\prec) = \Omega(Z)$. 
\end{center}
By induction on the ordinal rank of $Y$ ordered by $\prec$, we observe that this definition makes sense.  The proof is now done. 
\end{proof}
Finally, the following lemma is immediate from Lemma \ref{borkim2}. 
\begin{lemma}\label{borkim}
The functional $\Omega_{\rm b}$ is in the $\Omega$-cluster. Moreover, no functional in the $\Omega$-cluster is countably based.
\end{lemma}
We also note that Theorem \ref{weakk} establishes that functionals from the $\Omega$-cluster are weak in combination with $\exists^{2}$.

\subsection{Functionals related to bounded variation}\label{sams}
In this section, we identify a number of functionals in the $\Omega$ and $\Omega_{1}$-cluster based on basic properties of functions of \emph{bounded variation} introduced in Section \ref{defki}.  
We obtain similar results for the larger class of \emph{regulated} functions.  
We first introduce some required definitions (Section \ref{defzef}) and establish the associated equivalences (Section \ref{seccer}).

\subsubsection{Definitions}\label{defzef}
First of all, we introduce a notion of `realiser' for the Jordan decomposition theorem (Definition \ref{JDR}), as well as other functionals witnessing basic properties of functions of bounded variation (Definition \ref{JDR2}).  
Similar constructs exist in the literature: the proof of \cite{kreupel}*{Prop.\ 17} essentially shows that computing the total variation of a \emph{coded} $BV$-function amounts to computing the Turing jump.  
\begin{definition}[Functionals witnessing the Jordan decomposition theorem]\label{JDR}\rm~
\begin{itemize}
\item A \emph{Jordan realiser} is a partial functional $\J^{3}$ taking as input a function $f:[0,1] \rightarrow \R$ of bounded variation (item \eqref{donp} in Definition \ref{varvar}), and providing a pair $(g,h)$ of increasing functions $g$ and $h$ such that $f = g-h$ on $[0,1]$.
\item An \emph{intermediate Jordan realiser} is a partial functional $\J_{\ii}^{3}$ taking as inputs $f:[0,1] \rightarrow \R$ of bounded variation and an upper bound $k_{0}$ for the total variation (item \eqref{donp} in Definition \ref{varvar}), and providing a pair $(g,h)$ of increasing functions $g$ and $h$ such that $f = g-h$ on $[0,1]$.
\item A \emph{weak Jordan realiser} is a partial functional $\J_{\w}^{3}$ taking as inputs $f:[0,1] \rightarrow \R$ and its variation $V_{0}^{1}(f)$ (item \eqref{donp2} in Definition~\ref{varvar}), and providing a pair $(g,h)$ of increasing functions $g$ and $h$ such that $f = g-h$ on $[0,1]$.
\end{itemize}
\end{definition}
A weak Jordan realiser cannot compute a Jordan realiser in general; this remains true if we combine the former with an arbitrary type 2 functional (Corollary~\ref{frlom}).  

\smallskip

Secondly, we have shown in \cite{dagsamXII} that Jordan realisers are computationally equivalent to the following functionals. 
\begin{definition}[Functionals related to bounded variation]\label{JDR2}\rm~
\begin{itemize}
\item A \emph{$\SUP$-realiser} is a partial functional $\mathcal{S}^{3}$ taking as input a function $f:[0,1] \rightarrow \R$ which has bounded variation (item \eqref{donp} in Definition \ref{varvar}), and providing the supremum $\SUP_{x\in [0,1]}f(x)$. 
\item A \emph{continuity-realiser} is a partial functional $\mathcal{L}^{3}$ taking as input a function $f:[0,1] \rightarrow \R$ of bounded variation (item \eqref{donp} in Definition~\ref{varvar}), and providing an enumeration $(x_{n})_{n\in \N}$ of all points of discontinuity of $f$ on $[0,1]$. 
 \end{itemize}
\end{definition}
We note that continuity realisers are slightly different from e.g.\ Jordan realisers: while the latter always produce an output for \emph{any} $BV$-function, a continuity realiser has nothing to output in case of a \emph{continuous} input function (in $BV$).  
In this case, the enumeration of the empty set is of course the output (see Notation \ref{defornota}). 

\smallskip

Fourth, regulated functions yield various interesting functionals, as follows.
\bdefi[Functionals related to regularity]
\begin{itemize}
\item A \emph{Banach realiser} is a partial functional taking as input a regulated $f:[0,1] \rightarrow \R$ and providing the Banach indicatrix $N(f)$ as in \eqref{indix} as output. 
\item A \emph{Sierpi\'{n}ski realiser} is a partial functional $I^{3}$ which on input a regulated $f:[0,1]\di \R$ produces $I(f):=(g, h)$ such that $f=g\circ h$ with $g$ continuous and $h$ strictly increasing on their respective (interval) domains. 
\item A \emph{Baire-1-realiser} is a partial functional taking as input regulated $f:[0,1] \rightarrow \R$ and providing a sequence $(f_{n})_{n\in \N}$ of continuous $[0,1]\di \R$-functions that converges to $f$ on $[0,1]$. 
\end{itemize}
\edefi
\noindent
Recall that regulated functions are in the Baire class $1$, explaining the final item.

\smallskip

Finally, we shall need the following lemmas. 
The use of $\exists^{2}$ is perhaps superfluous in light of the constructive proof in \cite{tognog}, but the latter seems to make essential use of the Axiom of (countable) Choice.   
\begin{lemma}\label{korf} 
There is a functional $\D$, computable in $\exists^2$, such that if $f:[0,1] \rightarrow \R$ is monotone, then $\D(f)$ enumerates all points of discontinuity of $f$ on $[0,1]$. 
\end{lemma}
\begin{proof}
Immediate from \cite{dagsamXII}*{Lemma 7}.
\end{proof}
A set $C$ as in the following lemma is also called `RM-closed' as it is given by the coding of closed sets used in reverse mathematics (see \cite{simpson2}*{I-II}).
\begin{lemma}\label{korf2} 
There is a functional $\mathcal{E}$, computable in $\exists^2$, such that for any sequences $(a_{n})_{n\in \N}$, $(b_{n})_{n\in \N}$, if the closed set $C=[0,1]\setminus \cup_{n\in \N}(a_{n}, b_{n}) $ is countable, then $\mathcal{E}(\lambda n.(a_{n}, b_{n}))$ enumerates the points in $C$.
\end{lemma}
\begin{proof}
By \cite{harry}*{Theorem 2.12}, if the set $C$ is countable, then, all elements in $C$ are hyperarithmetical in  $(a_{n})_{n\in \N}$ and $(b_{n})_{n\in \N}$. 
Using Gandy selection, we can then find an enumeration of $C$ computable in $\exists^2$ and the parameters.  
We do not use that $C$ is closed, only that it is countable and hyperarithmetical in the parameters.
\end{proof}
Finally, we note that Dirichlet shows in \cite{didi1} that the Fourier series of certain \textbf{piecewise} continuous $f:[0,1]\di \R$ converge to $\frac{f(x+)+f(x-)}{2}$.  
As mentioned above, Jordan generalises this result to $BV$-functions in \cite{jordel}.  As expected in light of $\Omega_{\fin}$, many operations on piecewise continuous functions, like finding a maximum of supremum, are
part of the $\Omega$-cluster.  

\subsubsection{Computational equivalences}\label{seccer}
%
We identify a large number of inhabitants of the $\Omega$ and $\Omega_{1}$-clusters, including the functionals from Section \ref{defzef}.

\smallskip

In particular, we prove the following theorem where we note that we can define `intermediate' versions of all functionals pertaining to $BV$-functions, not just Jordan realisers. 
Similarly, `$BV$' can often be replaced by the weaker property `regulated', as is clear from items \eqref{contje} and \eqref{contje2} and Remark \ref{essenti}.  
As explored in Section \ref{flukkkk}, we may also often replace `$BV$' by the smaller Sobolev space $W^{1,1}$ or the pseudo-monotone functions.  
\begin{thm}[Second cluster theorem]\label{thm.surprise}
Assuming $\exists^{2}$, the following are computationally equivalent:
\begin{enumerate} 
\renewcommand{\theenumi}{\roman{enumi}}
\item a Jordan realiser, \label{jordje}
\item a $\SUP$-realiser,  \label{supje}
\item a continuity realiser,\label{contje}
\item an intermediate Jordan realiser, \label{jordje2}
\item an enumeration functional,\label{enu}
\item the functional $\Omega$,\label{fije}
\item a \textbf{quasi-}continuity realiser, i.e.\ a continuity realiser that only lists points of \emph{non-quasi-continuity} \(see Definition \ref{flung}\),\label{qcontje}
\item a \textbf{lower-semi-}continuity realiser, i.e.\ a continuity realiser that only lists points of \emph{non-lower-semi-continuity} \(see Definition \ref{flung}\),\label{qcontje2}
\item a functional $V^{3}$ such that $V(f, c, d)=V_{c}^{d}(f)$ for $[c, d]\subset [0,1]$ and $f:[0,1]\di \R$ of bounded variation \(item \eqref{donp} in Definition~\ref{varvar}\).\label{waard}
\item a functional $W^{3}$ such that $V_{c}^{d}(f)\leq W(f, c, d)$ for any $[c, d]\subset [0,1]$ and $f:[0,1]\di \R$ of bounded variation \(item \eqref{donp} in Definition~\ref{varvar}\).\label{waard2}
\item a functional $F^{3}$ such that $F(f)=(a_{n, m}, b_{n, m})_{n,m\in \N}$ for $f:[0,1]\di \R$ of bounded variation \(item \eqref{donp} in Definition~\ref{varvar}\) and set of discontinuities given by the $\textbf{\textup{F}}_{\sigma}$-set $\cup_{n\in \N} ( [0,1]\setminus \cup_{m\in \N} (a_{n,m}, b_{n,m}))$,\label{foefje}
\item the distance functional $d:(\R\times (\R\di \R)\times(\R\di \R))\di \R$, such that $d(x, A, Y)=\sup_{a\in A} |x-a|$, for $A\subset [0,1]$ and $Y$ injective on $A$, \label{katono}
\item a functional taking inputs $A\subset [0,1]$ and $Y:[0,1]\di \N$ injective on $A$ and outputting increasing $f:[0,1]\di \R$ discontinuous exactly at each $a\in A$.\label{rohim}
\item a continuity realiser for \emph{regulated} functions,\label{contje2}
\item a $\SUP$-realiser for \emph{regulated} functions,  \label{supje2}
\item a functional that on input any regulated $f:[0,1]\di \R$ outputs the pair $(\underline{f}, \overline{f})$ as in \eqref{confaged}, in case these envelopes are different, and $(f)$ otherwise.\label{envy}
\item a Sierpi\'nski realiser,\label{sierje2}
\item a Banach realiser, \label{baje2}
\item a Baire-$1$-realiser.  \label{B1}
\item a well-order realiser $\mathcal{Q}^{3}$ taking as input $B\subset A\subset [0,1]$, $Y:[0,1]\di \N$ injective on $A$, and $\preceq$ a well-ordering of $A$, and where $\mathcal{Q}(A, B, \preceq, Y)$ is the least element in $B$ relative to $F$, \label{QC}
\item a functional taking as input a regulated and upper semi-continuous $f:[0,1]\di \R$ and outputting $x\in [0,1]$ with $(\forall y\in [0,1])(|f(y)|\leq |f(x)|)$,\label{onetoomany}
\item \(Bolzano-Weierstrass\) a functional taking as input $A\subset \R$ without limit points and outputting $0$ \(resp.\ 1\) in case $A$ is finite \(resp.\ unbounded\).\label{twotoomany}
\end{enumerate}
\end{thm}
\begin{proof}
We establish the equivalences in the theorem following increasing item numbers.  
The equivalence between items \eqref{jordje}-\eqref{contje} may be found in \cite{dagsamXII}*{Theorem 3.4}.  
By the proof of \cite{dagsamXII}*{Theorem 3.9}, an intermediate Jordan realiser  computes an enumeration functional.   
Theorem \ref{fct} now yields $\Omega$, i.e.\ we have $\eqref{contje}\di \eqref{jordje2}\di \eqref{enu}\di \eqref{fije}$.
For \eqref{fije} $\di $ \eqref{contje}, Theorem \ref{fct} provides $\Omega_{\fin}$ and let $f:[0,1]\di \R$ be of bounded variation. 
If $x\in [0,1]$ is a point of discontinuity of $f$, the value $f(x)$ contributes a non-trivial amount to the variation, and we can measure `how much' by considering how $f(x)$ relates to the left and right limits $f(x-)$ and $f(x+)$.
This can be done in a computable way, in terms of $x, f$ and $ \exists^2$.
For each $k\in \N$, we may define $X_{k}$, the sets of points of discontinuity that provide a value larger than $\frac{1}{2^{k}}$ to the variation.  
Using $\Omega_{\fin}$, we can of course enumerate $X_{k}$, and taking the union, we may enumerate all points of discontinuity of $f$, i.e.\ a continuity realiser is obtained. 
Hence, we have already established the equivalences between items \eqref{jordje}-\eqref{fije}.

\smallskip

For \eqref{qcontje}$\di \eqref{fije}$, the indicator function $\mathbb{1}_{X}$ has bounded variation (with upper bound $|X|+1$) in case $X\subset [0,1]$ is finite. 
Any element of $X$ is a point where $\mathbb{1}_{X}$ is not quasi-continuous, i.e.\ $\Omega_{\fin}$ follows. 
For $\eqref{contje}\di \eqref{qcontje}$,  $f(x+)$ and $f(x-)$ are available at any $x\in (0,1)$.  Since $BV$-functions only have removable or jump continuities, we can use $\exists^{2}$ to check whether or not a
given point of discontinuity is also a point of non-quasi-continuity.  An analogous proof establishes the equivalence involving \eqref{qcontje2}. 
Hence, the equivalences between items \eqref{jordje}-\eqref{qcontje2} are ready.

\smallskip

For \eqref{waard} $\di $ \eqref{jordje}, $g(x):=\lambda x. V(f, 0, x)$ is well-defined and non-decreasing in case $f:[0,1]\di \R$ has bounded variation.   One proves that $h(x)=\lambda x. [V(f, 0, x)-f(x)]$ is monotone and then clearly $f=g-h$ on $[0,1]$. 
Hence, item \eqref{waard} yields a Jordan realiser.  If we have a Jordan realiser $\J$, then for $f:[0,1]\di \R$ of bounded variation with $\J(f)=(g, h)$ and $0\leq c< d\leq 1$, we have $V_{c}^{d}(f)=g(d)-g(c)+h(c)-h(d)$, using the usual `telescoping sum' trick.  A functional as in item \eqref{waard} is now immediate, i.e.\ $\eqref{jordje}\asa \eqref{waard}$ follows.    That $\eqref{waard}\di \eqref{waard2}$ is trivial, while $W$ as in the latter computes $\Omega_{\geq \#, \fin}$ and Theorem \ref{fct} yields $\Omega$.   
For the former claim, $\mathbb{1}_{X}$ has bounded variation for finite $X\subset [0,1]$ and $|X|=V_{0}^{1}(\mathbb{1}_{X})\leq W(\mathbb{1}_{X}, 0, 1)$.

\smallskip

For \eqref{contje} $\di $ \eqref{foefje}, a continuity realiser $\mathcal{L}$ outputs a sequence $\mathcal{L}(f)=(x_{n})_{n\in\N}$ which lists the points of discontinuity of $f:[0,1]\di \R$ of bounded variation.  
Each $\{x_{n}\}$ is trivially RM-closed as the complement of the open set $[ 0, x_{n})\cup (x_{n}, 1 ]$, readily yielding item \eqref{foefje}.  
To establish \eqref{foefje} $\di $ \eqref{contje} , we can use Lemma \ref{korf2}.  
Indeed, $C_{n}:= [0,1]\setminus \cup_{m\in \N} (a_{n,m}, b_{n,m})$ is an RM-closed and countable set if the double sequence is the output of item \eqref{foefje}.  
Hence, we may enumerate each $C_{n}$, and join all these together to obtain the output of a continuity realiser. 

\smallskip

Assuming item \eqref{katono}, note that for $A\subset [0,1]$ and $Y:[0,1]\di \N$ injective on $A$:
\be\label{basik}\textstyle
(\forall x\in (\frac12, 1])(x\not \in A)\asa (\forall q\in \Q\cap (\frac12, 1])( d(q, A, Y)> |q-\frac12|),
\ee
and similar for any interval replacing $(\frac12, 1]$.  Hence, the usual `interval-halving' technique finds the supremum of $A$ using $\exists^{2}$ and \eqref{basik}, yielding $\Omega_{\BW}$ and hence $\Omega$ by Theorem~\ref{fct}.  To obtain item \eqref{katono} from an enumeration functional, the latter converts the supremum in the former to one over $\N$.
The first twelve items of the theorems have now been shown to be equivalent.  

\smallskip

For \eqref{enu}$\di$\eqref{rohim}$\di$\eqref{contje}, use Lemma \ref{korf} for the second implication while the first implication follows from considering $f(x):= \sum_{x_{n}\leq x}\frac{1}{2^{n}}$ where $(x_{n})_{n\in \N}$ is any enumeration of the countable set $A\subset[0,1]$.

\smallskip

For item \eqref{contje2}, the latter implies \eqref{contje} as $BV$-functions are regulated.   For the equivalence, we obtain item \eqref{contje2} from $\Omega_{\fin}$.  To this end, define 
\be\label{cuntila}\textstyle
X_{k}:=\big\{x\in (0,1): |f(x+)- f(x)|>\frac1{2^{k}} \vee |f(x-)- f(x)|>\frac1{2^{k}}\big\} 
\ee
which collects the finitely many\footnote{Note that if some $X_{k}$ is not finite, it has a cluster point $x_{0}\in[0,1]$; however then either $f(x_{0}+)$ or $f(x_{0}-)$ does not exist, a contradiction.} points of discontinuity of the regulated function $f$ where the jump is in excess of $1/2^{k}$.  
The union over $n\in \N$ of all finite sequences $\Omega_{\fin}(D_{n})$ then enumerates the points of discontinuity of $f$, as required for item \eqref{contje2}. 
Then $\eqref{contje}\asa \eqref{contje2}\asa\eqref{supje}\asa \eqref{supje2}$ is immediate. 

\smallskip

For item \eqref{envy}, consider the definition of upper and lower envelope in \eqref{confaged} and note that the (inner) supremum and infimum over $\R$ can be replaced 
by a supremum and infimum over $\N$ (and $\Q$), if we have access to a sequence listing all points of discontinuity of a regulated function $f:[0,1]\di \R$.  
Moreover, continuity of $f$ at a point $x\in [0,1]$ implies that $\overline{f}(x)=\underline{f}(x)$. 
Hence, we can check whether $\underline{f}=\overline{f}$ using $\exists^{2}$, and \eqref{contje2}$\di $ \eqref{envy} follows. 
Now assume \eqref{envy} and consider finite $X\subset \N$.  Then for $f=\mathbb{1}_{X}$, we have that $\underline{f}=\overline{f}$ everywhere if and only if $X=\emptyset$.  
In case $X\ne \emptyset$, the usual interval-halving technique can locate a point therein, i.e.\ $\Omega$ follows.  

\smallskip

For \eqref{sierje2} $\di$ \eqref{contje2}, given a Sierpi\'nski realiser, the second component of the output $I(f)=(g,h)$ is (strictly) monotone and we can enumerate the points of discontinuity of $h$ by Lemma \ref{korf}.
Now exclude all points $x\in [0,1]$ from this sequence for which $f(x+)=f(x-)$, which can be done using $\exists^{2}$.   
In this way, we obtain a continuity realiser as in \eqref{contje2}. 
To derive a Sierpi\'nski realiser from a continuity realiser, i.e.\ \eqref{contje2} $\di $ \eqref{sierje2}, we fix regulated $f:[0,1]\di \R$ and consider the proof of \cite{voordedorst}*{Theorem 0.36, p.\ 28}, going back to \cite{voordesier}. 
This proof establishes the existence of $g, h$ such that $f=g\circ h$ with $g$ continuous and $h$ strictly increasing. 
Moreover, one finds an \emph{explicit construction} (modulo $\exists^{2}$) of the function $h$ required, assuming a sequence listing all points of discontinuity of $f$ on $[0,1]$.  
The function $g$ is then defined as $\lambda y.f(h^{-1}(y))$ where $h^{-1}$ is the inverse of $h$, definable using $\exists^{2}$. 
In this light, a continuity realiser plus $\exists^{2}$ yields a a Sierpi\'nski realiser. 

\smallskip

For \eqref{sierje2} $\di$ \eqref{baje2}, Banach's proof of \cite{banach1}*{Theorems 1 and 2} essentially establishes that $\exists^{2}$ computes a Banach realiser in case the function at hand is \emph{additionally} {continuous} on $[0,1]$.  
Banach's results from \cite{banach1} are also published in English in e.g.\ \cite{naatjenaaien}*{p.\ 225}.
For the general case, as discussed in \cite{voordedorst}*{p.~44}, the Banach indicatrix $N(f)$ equals $N(g)$ if $g$ is continuous and satisfies $f=g\circ h$ for some strictly increasing $h$.   
Hence, a Sierpi\'nski realiser readily yields a Banach realiser.  

\smallskip

Next, we show how to compute an enumeration functional from a Banach realiser.  Hence, fix $A\subset [0,1]$ and let $Y:[0,1]\di \N$ be injective on $A$.  
Now define the following function using $\exists^{2}$:
\be\label{klamdake}
f_{q}(x):=
\begin{cases}
\frac{1}{2^{Y(x)+1}} & x\in A\wedge x>_{\R}q \\
0 &\textup{ otherwise }
\end{cases}.
\ee
The following equivalence is then readily proved, for any $n\in \N$ and $q\in\Q\cap [0,1]$:
\be\label{south}\textstyle
(\exists x\in A)(Y(x)=n \wedge x>_{\R}q)\asa [1=N(f_{q})(\frac{1}{2^{n+1}})].
\ee
Using \eqref{south} and $\exists^{2}$, we can decide if for $n\in \N$ there is (unique) $x_{0}\in A$ such that $Y(x)=n$.  
If such $x_{0}$ exists, we can successively approximate it using the usual interval-halving technique, using again $\exists^{2}$ and \eqref{south}.  
In this way, we can obtain a sequence $(x_{n})_{n\in \N}$ listing all elements of $A$, i.e.\ \eqref{baje2} $\di $ \eqref{enu} is complete. 
Hence, the first 18 items are equivalent. 

\smallskip

For \eqref{QC}, an enumeration functional can enumerate a countable $A\subset [0,1]$ given some $Y:[0,1]\di \N$ injective on $A$.  
Moreover, given an enumeration of $A$, an unbounded search readily yields the least element of any $B\subset A$, as required by a well-order realiser, i.e.\ \eqref{enu} $\di $ \eqref{QC} follows. 
Now assume a well-order realiser $\mathcal{Q}$ is given and fix a countable $A\subset [0,1]$ and $Y:[0,1]\di \N$ injective on $A$.  
Define a well-order on $A$ by $x\preceq y$ if and only if $Y(x)\leq Y(y)$.  
Define $B_{n}:=\{x\in A : Y(x)\geq n\}$, $x_{0}:=\mathcal{Q}(A, B_{0}, \preceq, Y)$, and $x_{m+1}:=\mathcal{Q}(A, B_{Y(x_{m})}, \preceq, Y)$.  Clearly, this sequence readily yields an enumeration of $A$, i.e.\ \eqref{QC} $\di$ \eqref{enu}.

\smallskip

Regarding item \eqref{B1}, $\exists^{2}$ readily computes the supremum (or maximum) of any continuous function on an interval by \cite{kohlenbach2}*{\S3}.  
Hence, the approximation provided by item \eqref{B1} readily yields the supremum required by item \eqref{supje2}.  Now item \eqref{contje}, provides the points of discontinuity of regulated functions.  
Using $\exists^{2}$, one then readily defines the sequence $(f_{n})_{n\in \N}$ required by item \eqref{B1}.  

\smallskip

Assuming item \eqref{onetoomany}, for finite $X\subset [0,1]$, the function$\mathbb{1}_{X}$ is regulated and upper semi-continuous.  
The maximum provided by item \eqref{onetoomany} allows us to decide if $X=\emptyset$ or not.  In the latter case, we also
find an element of $X$, yielding $\Omega$ as in item \eqref{fije}.  Assumie \eqref{contje2} and note that a maximum of $f$ -as required by \eqref{onetoomany}- 
is either a point of continuity of $f$ or in the sequence provided by \eqref{contje2}.  
In the former case, we may approximate it using rationals, i.e.\ we have $\eqref{contje2}\di \eqref{onetoomany}$, and we are done.

\smallskip

To obtain item \eqref{twotoomany} from $\Omega$, let $A\subset \R$ be a set without limit points.  
Then $A\cap [-n, n]$ is finite for any $n\in \N$, i.e.\ $\Omega_{\fin}$ (see Theorem \ref{fct}) can enumerate $A$, allowing us to decide whether this set is unbounded or finite.  
To show that item~\eqref{twotoomany} computes $\Omega_{\textsf{b}}$, let $X\subset[0,1]$ be finite.  Now define $Y\subset \R$ as $Y:=\{ y\in \R : (\exists n\in \N)(|y-n|\in X  )\}$.
In case $X$ is empty (resp.\ a singleton), the set $Y$ must be finite (resp.\ unbounded), i.e.\ $\Omega_{\textsf{b}}$ readily follows. 
\end{proof}
Regarding item \eqref{twotoomany}, Weierstrass formulates the `Bolzano-Weierstrass theorem' around 1860 in \cite{weihimself}*{p.~77} as follows, while Bolzano \cite{russke}*{p.~174} states the existence of suprema rather than just limit points.  
\begin{quote}
If a function has a definite property infinitely often within a finite domain, then there is a point such that in any neighbourhood of this point there are infinitely many points with the property.
\end{quote}
We note that item \eqref{twotoomany} witnesses the contraposition of Weierstrass' theorem.  

\smallskip

Next, the equivalences in the theorem are robust in the following sense: we could replace `$BV$-function' in Theorem \ref{thm.surprise} by `$BV$-function that comes with a second-order code as in \cite{kreupel}*{Def.\ 1}' and 
the proof would still go through.  In particular, it seems the \emph{presence} of codes (in the sense of \cite{kreupel}) for third-order $BV$-functions does not change the computational properties as listed in Theorem \ref{thm.surprise}.
This is no surprise as a code in the sense of \cite{kreupel} only describes an ($L_{1}$-)equivalence class of $BV$-functions, not an individual $BV$-function.

\smallskip

Next, in light of the equivalence between items \eqref{contje} and \eqref{contje2}, one can replace `$BV$' or `regulated' in these items by \textbf{any} intermediate class, discussed next.  
\begin{rem}[Between bounded variation and regulated]\label{essenti}\rm 
The following spaces are intermediate between $BV$ and regulated; all details may be found in \cite{voordedorst}.  

\smallskip

Wiener spaces from mathematical physics (\cite{wiener1}) are based on \emph{$p$-variation}, which amounts to replacing `$ |f(x_{i})-f(x_{i+1})|$' by `$ |f(x_{i})-f(x_{i+1})|^{p}$' in the definition of variation \eqref{tomb}. 
Young (\cite{youngboung}) generalises this to \emph{$\phi$-variation} which instead involves $\phi( |f(x_{i})-f(x_{i+1})|)$ for so-called Young functions $\phi$, yielding the Wiener-Young spaces.  
Perhaps a simpler construct is the Waterman variation (\cite{waterdragen}), which involves $ \lambda_{i}|f(x_{i})-f(x_{i+1})|$ and where $(\lambda_{n})_{n\in \N}$ is a Waterman sequence (of reals with nice properties); in contrast to $BV$, any continuous function is included in the Waterman space (\cite{voordedorst}*{Prop.\ 2.23}).  Combining ideas from the above, the \emph{Schramm variation} involves $\phi_{i}( |f(x_{i})-f(x_{i+1})|)$ for a sequence $(\phi_{n})_{n\in \N}$ of well-behaved `gauge' functions (\cite{schrammer}).  
As to generality, the union (resp.\ intersection) of all Schramm spaces yields the space of regulated (resp.\ $BV$) functions, while all other aforementioned spaces are Schramm spaces (\cite{voordedorst}*{Prop.\ 2.43 and 2.46}).
In contrast to $BV$ and the Jordan decomposition theorem, these generalised notions of variation have no known `nice' decomposition theorem.  The notion of \emph{Korenblum variation} (\cite{koren}) does have such a theorem (see \cite{voordedorst}*{Prop.\ 2.68}) and involves a distortion function acting on the \emph{partition}, not on the function values (see \cite{voordedorst}*{Def.\ 2.60}).  
\end{rem}
It is no exaggeration to say that there are \emph{many} natural spaces between the regulated and $BV$-functions, all of which yield natural functionals for the $\Omega$-cluster.  
A non-trivial example is any functional that on input a regulated $[0,1]\di \R$-function outputs the associated Waterman sequence and upper bound for the Waterman variation.  
By the proof of \cite{voordedorst}*{Prop.\ 2.24} (and associated lemmas), a Sierpi\'nski realiser readily computes such a functional.  

\smallskip

Finally, the $\Omega$-cluster is defined in terms of (an equivalent formulation of) S1-S9-computability.  The following results show that one can get by with a weaker notion of computability.   
In particular, the following theorem shows how to obtain an injection for the countable set of discontinuities of regulated functions. 
\begin{thm}
There is a term $t$ of G\"odel's $T$ such that for any regulated $f:[0,1]\di \R$, the mapping $\lambda x.t(x, f, \exists^{2}, \Omega_{\geq \#, \fin})$ is a $[0,1]\di \N$-function injective on 
\[
D_{f}:=\{ x\in [0,1]:f(x+)\ne f(x) \vee f(x)\ne f(x-) \}, 
\] 
which is the countable set of discontinuities of $f$.
\end{thm}
\begin{proof}
As noted in the proof of Theorem \ref{thm.surprise}, the set $X_{n}$ as in \eqref{cuntila} must be finite for fixed $n\in \N$.
Hence, the following set is also finite:  
\[\textstyle
Y_{n}:=\{  (x, y, k)\in \R^{2}\times \N:  x, y\in X_{n}\wedge 0< |x-y|\leq\frac{1}{2^{k}} \}. 
\]
Then $g(n):=\Omega_{\geq\#, \fin}(Y_{n})+1$ is such that $(\forall x, y\in X_{n})( |x-y|>\frac{1}{2^{g(n)}}   )$.
In other words, we may not know the elements of $X_{n}$, but we do know how much they must be apart (at least). 
Hence, for $x\in X_{n}$, the ball $B(x, \frac{1}{2^{g(n)+2}})$ does not contain any other elements of $X_{n}$.  
Let $(q_{m})_{m\in \N}$ be an enumeration of the rationals in $[0,1]$.
Then we may choose a rational, say with index $Z_{n}(x)$,  in $B(x, \frac{1}{2^{g(n)+2}})$ such that $Z_{n}(x)\ne Z_{n}(y)$ for $x, y\in X_{n}$ and $x\ne y$.
Now define $Y:[0,1]\di \N$ as follows:
\[
Y(x):= 
\begin{cases}
2^{n} \times p_{(1+Z_{n}(x)) }& x \in X_{n} \textup{ and $n$ is the least such number}\\
0 & \textup{otherwise}
\end{cases},
\]
where $p_{n}$ is the $n$-th prime number.  
This mapping is injective on $\cup_{n\in \N}X_{n}$ and is definable as required by the theorem.  
\end{proof}
\begin{cor}\label{poil}
There is a term $s$ of G\"odel's $T$ such that for regulated $f:[0,1]\di \R$,  $s( f, \exists^{2}, \Omega_{\geq \#, \fin})$ provides an enumeration of the points of discontinuity of $f$.
\end{cor}
\begin{proof}
By the (rather effective) last part of the proof of Theorem \ref{fct}, we can compute $\Omega$ in the restricted sense of the corollary.  The functional $\Omega$ readily (and similarly effectively) yields an enumeration functional by considering $\Omega(  \{ x\in [0,1]: x \in A\wedge Y(x)=n  \} )$ for $A\subset [0,1]$ and $Y:[0,1]\di \N$ injective on $A$.
\end{proof}
In conclusion, despite Corollary \ref{poil}, we believe that the correct definition of the $\Omega$-cluster involves the generality of S1-S9.  
We also believe there to be natural functionals in the $\Omega$-cluster for which the analogue of Corollary \ref{poil} is not possible.  

\subsection{On the weakness of weak Jordan realisers}\label{separ}
In this section, we show that weak Jordan realisers are deserving of their name: we show that the latter cannot, in general, compute a Jordan realiser, even when combined  with an arbitrary type 2 functional.  
To this end, we establish the following theorem. 
\begin{thm}\label{jordomega1}
The class of weak Jordan realisers is in the $\Omega_1$-cluster.
\end{thm}
\begin{proof}
We first prove that $\Omega_1$ is computable in any weak Jordan realiser. Let $X \subseteq [0,1]$ have exactly one element and let $f$ be the characteristic function of $X$. 
Then exactly one of three will be the case: $0 \in X$, $1 \in X$ or $V_0^1(f) = 2$.
In the latter case, we use the decomposition $f = g - h$ obtained from the weak Jordan realiser, enumerate all points of discontinuity of $g$ and $h$ using $\exists^2$, and search for the one element in $X$ among those points of discontinuity.

\smallskip

Next, we prove that there is a weak Jordan realiser computable in $\Omega_1$ and $\exists^2$.  Let $f:[0,1]\di \R$ be given with known variation $a\in \R$. 
If $(x_n)_{n \in \N}$ is a sequence in the unit interval, we can define $V(f,(x_n)_{n \in \N})$ as the supremum  we obtain by restricting all partitions to elements in the sequence (using $\exists^{2}$). 
We then let $X$ be the set of sequences $(x_n)_{n \in \N}$ such that $V(f,(x_n)_{n \in \N}) = a$ and such that all rational numbers in $[0,1]$  appear in the sequence.

\smallskip

Now, if $(x_n)_{n \in \N} \in X$, all discontinuity points $x$ of $f$, where $f(x)$ is not in the closed interval between the left- and right limits of $f$ at $x$, will be in the sequence. 
We can then use partitions from any sequence in $X$ to construct a decomposition of $f$ into the increasing  $g^*$ and $h^*$. These will be correct except possibly for points of discontinuity where $f$ takes a value in the gap between the left and right limits. These points are however points of discontinuity of $g^*$ or $h^*$, and we can use $\exists^2$ to check if there are any by Lemma \ref{korf}.

\smallskip

We now let $Y$ be the set of enumerations $(x_n)_{n \in \N}$ where no such points of discontinuity are missing. For $(x_n)_{n \in \N} \in Y$, we will have that $g^*, h^*$ is the best possible decomposition, i.e. the decomposition where the growth of the increasing functions is the least possible. In this way, $g^{*}, h^{*}$ are independent of the choice of $(x_n)_{n \in \N}$.  Thus, there is an enumeration $(y_n)_{n\in N}$ of all points of discontinuity of $g^*$ and $h^*$, together with the rationals, computable from $(x_n)_{n \in \N} \in Y$ and $\exists^2$, but independent of the actual choice of sequence. 
Hence, the set $Z$ of sequences in $Y$ producing itself this way will contain \emph{exactly one} element, and we can use $\Omega_1$ to find it.
We then get the weak Jordan realiser by using $g^*,h^*$ constructed from this unique sequence.
\end{proof}
\begin{cor}\label{frlom}
A weak Jordan realiser cannot compute a Jordan realiser in general, even when combined with with an arbitrary type 2 functional. 
\end{cor}
\begin{proof}
As noted under Definition \ref{defomega1}, $\Omega_1$ is countably based, while $\Omega$ is not by Lemma \ref{borkim}.  Corollary \ref{kink} now finishes the proof. 
\end{proof}
The following theorem is now proved in analogy with the proof of  Theorem \ref{thm.surprise}.
Following Definition \ref{JDR}, any functional defined on $BV$ has a `weak' counterpart which has the variation \eqref{tomb} as an additional input.  
\begin{thm}[Third cluster theorem]\label{thm.surprise3}
Assuming $\exists^{2}$, the following are computationally equivalent:
\begin{itemize} 
\item the functional $\Omega_1$
\item a weak Jordan realiser, 
\item a weak continuity realiser,
\item a weak Sierpi\'nski realiser, that is, a Sierpi\'nski realiser restricted to $BV$-functions with known variation, 
\item a weak Banach realiser, that is, a Banach realiser restricted to $BV$-functions with known variation.
\item a weak enumeration functional,
\item a weak $\Omega_{\BW}$-functional.
\end{itemize}
\end{thm}
\begin{proof}
The computational equivalence of $\Omega_1$, a weak Jordan realiser, and a weak continuity realiser modulo $\exists^2$ follows from (the proof of) Theorem \ref{jordomega1}.
Moreover, it is straightforward to prove that $\Omega_1$, the weak $\Omega_{\BW}$-functional, and a weak enumeration functional (computing the inverse of a surjection from a set $X$ to $\N$) are computationally equivalent modulo $\exists^2$.
Finally, the computational equivalence of a weak Sierpi\'nski realiser, a weak Banach realiser, and a weak continuity realiser, 
can be proved in the same way as for the `non-weak' case, namely as in the proof of Theorem \ref{thm.surprise}.
\end{proof}
Finally, we identify one functional we believe to be strictly intermediate between the $\Omega$ and $\Omega_{1}$-functionals, based on K\"onig's (original) lemmas\footnote{The names \emph{K\"onig's infinity lemma} and \emph{K\"onig's tree lemma} are used in \cite{wever} which contains a historical account of these lemmas, as well 
as the observation that they are equivalent; the formulation of K\"onig's lemma involving trees apparently goes back to Beth around 1955 in \cite{bethweter}, as also discussed in detail in \cite{wever}.} (see \cite{koning147, koning26, koning16}). 
\begin{princ}[K\"onig's lemma $\KL$ for real numbers]\label{KIL2real}
Let $(E_{n})_{n\in \N}$ be a sequence of sets in $\R$ and let $R$ a binary relation on reals such that for all $n\in\N$ we have:
\begin{itemize}
\item the set $E_{n}$ is finite and non-empty,
\item for any $x\in E_{n+1}$, there is at least one $y\in E_{n}$ such that $yRx$.
\end{itemize}
Then there is a sequence $(x_{n})_{n\in \N}$ such that for all $n\in \N$, $x_{n}\in E_{n}$ and $x_{n}Rx_{n+1}$.
\end{princ}
\bdefi
A \emph{K\"onig realiser} takes as input a sequence of sets $(E_{n})_{n\in \N}$ and relation $R$ as in $\KL$ and outputs a sequence as in the latter.  
\edefi
\noindent
Clearly, the functional $\Omega$ computes a K\"onig realiser, which in turns computes $\Omega_{1}$.

\subsection{Functionals related to absolute continuity}\label{maink3}
We identify a number of interesting elements in the $\Omega$-cluster based on the notion of \emph{absolute continuity}, introduced in 
Section \ref{KLM}.  As is clear from Sections \ref{hu1}-\ref{flukkkk}, our study deals with both basic notions (arc length and constancy) and fundamental results (fundamental theorem of calculus and Sobolev spaces).
\subsubsection{Introduction}\label{KLM}
In this section, we introduce some required notions, like absolute continuity, and how differentiability applies to $BV$-functions. 
We tacitly assume all these notions pertain to $[0,1]$, unless explicitly stated otherwise. 

\smallskip

First of all, we make our notion of differentiability precise, as the latter is intimately related to $BV$.
Indeed, combining Jordan's decomposition theorem (Theorem \ref{drd}) and Lebesgue's theorem on the differentiability of monotone functions (see \cite{botsko2}*{Theorem~1} for an elementary proof), a $BV$-function is differentiable almost everywhere (a.e.\ for short).  
Moreover, a function that has a bounded derivative on $[0,1]$, is also a $BV$-function; the latter fact is already mentioned by Jordan right after introducing $BV$-functions (see \cite{jordel}*{p.\ 229}).  
We will bestow the following meaning on `derivative of a $BV$-function'.  
\begin{defi}\label{kood}
A \emph{derivative} of $f \in BV$ is any function $g:[0,1]\di \R$ such that $\lim_{h\di 0}\frac{f(x+h)-f(x)}{h}=g(x)$ for almost all $x\in [0,1]$.
\end{defi}
Since our notion of derivative is `almost unique', we shall abuse notation and refer to `the' derivative $f'$ of $f\in BV$.
Definition \ref{kood} also provides a motivation for focusing on $BV$ as the notion of derivative is automatically well-defined.  

\smallskip

Next, we define an interesting subclass $AC$ of $BV$ which additionally satisfies $AC\subset [C\cap BV]$ where $C$ is the class of all continuous functions on $[0,1]$.
The reader will know the existence of continuous functions\footnote{The standard example is $f:[0,1]\di \R$ defined as $0$ if $x=0$ and $x \cdot \sin(1/x)$ otherwise;  
Weierstrass' `monster' function is also a natural example of a continuous function not in $BV$.} not in $BV$.
\begin{defi}
A function $f:[0,1]\di \R$ is \emph{absolutely continuous} \(`$f\in AC$' for short\) if for every $k\in \N$, there is $N\in \N$  such that if a finite sequence of pairwise disjoint sub-intervals $(x_{i},y_{i})$ of $[0,1]$ with 
${ x_{i}<y_{i}\in [0,1]}$ satisfies ${ \sum _{i}(y_{i}-x_{i})<\frac{1}{2^{N}} }$, then we have ${ \sum _{i}|f(y_{i})-f(x_{i})|<\frac{1}{2^{k}}}$.
\end{defi}
Next, to avoid confusion (but perhaps not pedantry), we briefly recall a technicality related to continuity realisers.  
As noted below Definition \ref{JDR2}, Jordan realisers are defined for any $BV$-function while continuity realisers only make direct sense for discontinuous $BV$-functions; for continuous $BV$-functions, the output of a continuity realiser is the enumeration of the empty set as in Notation \ref{defornota}.

\smallskip

Finally, the following results will be used often.  Interestingly, $\exists^{3}$ is equivalent to a functional that decides whether \emph{arbitrary} $[0,1]\di \R$-functions are continuous. 
\begin{thm}\label{flagrant}
The following is in the $\Omega$-cluster: a functional deciding whether $f\in BV$ is continuous on any $[a,b]\subset [0,1]$.
\end{thm}
\begin{proof}
A continuity realiser (see Definition \ref{JDR2}) can obviously decide whether $f\in BV$ is continuous or not, namely by using $\exists^{2}$ to check whether the output of the former is the sequence $\langle\rangle*\langle \rangle*\dots$ or not.
For the reverse direction, a finite set $X\subset [0,1]$ has characteristic function $\mathbb{1}_{X}$ in $BV$.  Hence, the latter is continuous
on $[a,b]\subset [0,1]$ if and only $X\cap [a, b]=\emptyset$.  In this way, the usual interval-halving technique allows us to find the left-most point of $X$.  
Now remove this point from $X$ and repeat the same procedure until $X$ is empty.  
\end{proof}
The reader will verify that Theorem \ref{flagrant} remains true for `continuous' replaced by `quasi-continuous' or `lower semi-continuous' and/or `$BV$' replaced by `regulated'. 
Since any regulated function is cliquish, we cannot go `much below' quasi-continuity in Theorem \ref{flagrant}. 
We also have the related following theorem. 
\begin{thm}\label{flagrant2}
The following is in the $\Omega$-cluster: a functional deciding whether a regulated $f:[0,1]\di \R$ has bounded variation on any $[a,b]\subset [0,1]$.
\end{thm}
\begin{proof}
Given a regulated function $f:[0,1]\di \R$, we use $\Omega$ and $\exists^2$ to enumerate the points of discontinuity of $f$ (Theorem \ref{thm.surprise}).  In this case, the supremum in \eqref{tomb} can be replaced with a supremum over $\N$ (and $\Q$).  Hence, we can use $\exists^2$ to decide if the variation of $f$ is finite or infinite.  
To prove the other direction, we consider $\Omega_{\rm b}$ from the $\Omega$-cluster (see Section \ref{struc}).  
Let $I_n = [\frac{1}{2n+1},\frac{1}{2n}]$ and let $\phi_n$ be the affine bijection from $I_n$ to $[0,1]$.
Let $X  \subseteq [0,1]$ have at most one element. Define $f(x) :=  \frac{1}{n}$ if $x \in I_n$ and $\phi_n(x) \in X$, and 0 otherwise. Then $f$ is always regulated, and it is additionally of bounded variation if and only if $X$ is empty.
Hence, we obtain the functional $\Omega_{\rm b}$ and we are done. 
\end{proof}

\subsubsection{Connecting differentiability and constancy}\label{hu1}
In this section, we investigate the computational properties of well-known results pertaining to differentiability almost everywhere. 

\smallskip

It is hard to overstate the importance and central nature of derivatives to mathematics, physics, and engineering.  
In particular, derivatives provide essential qualitative information about the graphs of functions.  
For example, \emph{Fermat's theorem} implies that the local extremum of a differentiable function has zero derivative.  
Moreover, a {continuously} differentiable function is constant if and only if if it has zero derivative everywhere.  

\smallskip

However, the notion of a.e.\ differentiability, while seemingly close to the `full' notion, behaves quite differently.  Indeed, a function is called \emph{singular} (see e.g.\ \cite{gordonkordon}) if 
it has zero derivative a.e.\ and a continuous and singular function need not\footnote{Cantor's singular function, aka the \emph{Devil's staircase}, is a famous counterexample (see \cite{voordedorst, gordonkordon})} be constant.  
In particular, plain continuity is too weak to guarantee constancy for singular functions, while $AC$ suffices as follows. 
\begin{thm}[\cite{voordedorst}*{Prop.\ 3.33}]
A function on the unit interval is constant if and only if it is in $AC$ and singular. 
\end{thm}
By contraposition, we obtain the functional as in the following theorem. 
\begin{thm}\label{iop3}
The following is in the $\Omega$-cluster:  a functional that on input singular $f \in BV$ outputs $x, y\in [0,1]$, which satisfy $f(x)\ne f(y)$ in case $f\not \in AC$.
\end{thm}
\begin{proof}
For the forward direction, a continuity realiser outputs a sequence $(x_{n})_{n\in \N}$ enumerating the points of discontinuity of any $BV$-function $f:[0,1]\di \R$, if such points there are.
For $z=x_{n}$ with $f(z+)\ne f(z-)$, we readily find $x, y\in [0,1]$ such that $f(x)\ne f(y)$ by letting $x$ (resp.\ $y$) be a close enough approximation of $z$ from the left (resp.\ the right).  
As similar approach works in case $f(z-) = f(z+) \neq f(z)$.
In case $f$ is continuous, a continuity realiser will output the null sequence.  By \cite{kohlenbach2}*{Prop.\ 3.14}, $\exists^{2}$ can then compute a point $x\in [0,1]$ (resp.\ $y\in [0,1]$) where $f$ attains its maximum (resp.\ minimum).  Since $f$ is not constant, we have $x\ne y$ as required by the theorem. 

\smallskip

For the reverse direction, let $X\subset [0,1]$ be finite and consider $\mathbb{1}_{X}$, which is in $BV$. 
We may assume $X\cap \Q=\emptyset$ as $\mu^{2}$ can enumerate all rationals in $X$.  
Clearly, $\mathbb{1}_{X}$ has zero derivative a.e.\ and for any $x, y\in [0,1]$ with $\mathbb{1}_{X}(x)\ne \mathbb{1}_{X}(y)$, either $x\in X$ or $y\in X$, in case $X\ne \emptyset$.  
Now remove the thus obtained element from $X$ and repeat the procedure until $X$ is empty. 
Note that in case $X=\emptyset$, the reals $x, y\in [0,1]$ produced by the functional in the theorem are such that $\mathbb{1}(x)=\mathbb{1}(y)$, which is decidable given $\exists^{2}$.
\end{proof}
The previous result has a certain robustness, as follows. 
For instance, we can replace $ AC$ in the theorem by any class $ B$ intermediate between $AC$ and $C$.
A non-trivial example of such class $B$ is given by the Darboux\footnote{A function is called \emph{Darboux} if it satisfies the intermediate value property, i.e.\ for any $x, y$ in the domain of $f$ and $z$ such that $f(x)\leq z\leq f(y)$, there is $u$ such that $f(u)=z$.} $BV$-functions (see \cite{voordedorst}*{p.\ 78}).  Similarly, one can replace $AC$ in the theorem by the `zero variation space' $CBV$ by \cite{voordedorst}*{Prop.\ 1.20} or by the class of Darboux functions.

\subsubsection{The fundamental theorem of calculus}\label{hu2}
In this section, we study the computational properties of the second fundamental theorem of calculus.

\smallskip

First of all, the fundamental theorem of calculus (FTC for short) expresses that differentiation and integration cancel out, going back to Newton and Leibniz.      
It is a central question of analysis to which functions (various versions of) FTC applies.  We shall focus on the \emph{second} FTC, i.e.\ the integration of derivatives as in:
\be\label{fromk}\textstyle
f(b)-f(a)=\int_{a}^{b}f'(x)~dx,
\ee
and the general question for which functions $f:\R\di \R$ (and integrals) and $a, b\in [0,1]$ the equality \eqref{fromk} holds. 
As noted above, $BV$-functions are differentiable a.e.\ and bounded, i.e.\ the Lebesgue integral $\int_{0}^{1}f'(x)~dx$ exists (\cite{voordedorst}*{p.\ 222}).  
Nonetheless, the second FTC does not hold for $BV$ (see \cite{voordedorst}*{p.\ 252}) while it does hold for $AC$.
In particular, the following theorem is aptly called the \emph{fundamental theorem of calculus for the Lebesgue integral} in \cite{voordedorst}*{p.\ 222}.
\begin{thm}\label{liko}
A function $f:[0,1]\di \R$ is in $AC$ if and only if 
\begin{itemize}
\item $f$ is differentiable a.e.\ on $[0,1]$, 
\item the derivative $f'$ is in $L_{1}$, 
\item $f(b)-f(a)=\int_{a}^{b}f'(x) ~dx $ for any $0\leq a\leq b\leq 1$. 
\end{itemize}
\end{thm}
\noindent
Note that we always have $f(b)-f(a)\geq \int_{a}^{b}f'(x) ~dx $ for $f\in BV$ by \cite{voordedorst}*{p.\ 238}.
By the above, we can replace the first item in Theorem \ref{liko} by `$f\in BV$' or $`f\in BV\cap C$'.  
By contraposition, Theorem \ref{liko} then yields the following `$\forall \exists$' statement:
\begin{center}
for $f\in BV\setminus AC$ with integrable derivative $f'$, there are $a, b\in [0,1]$ for which FTC fails, i.e.\ the third item in Theorem \ref{liko} is false.     
\end{center}
This `$\forall \exists$' statement gives rise to the functional in Theorem \ref{iop}, where we use the following definitions.  
The notion of `effectively Riemann integrable' is well-known\footnote{A function is \emph{effectively Riemann integrable} on $[0,1]$ if a `modulus' function is given which on input $\eps>0$ produces $\delta>0$ satisfying the usual `$\eps$-$\delta$' definition of Riemann integrability.} in RM and computability theory (see e.g.\ \cite{sayo});
we could use some effective version of the Lebesgue integral instead.  
\begin{thm}\label{iop}
The following is in the $\Omega$-cluster:
a functional that on input $f \in BV$ with effectively Riemann integrable derivative $f'$ outputs $y \in [0,1]$, which satisfies $f(y)-f(0)> \int_{0}^{y}f'(x) ~dx $ in case $f\not \in AC$.
\end{thm}
\begin{proof}
For the forward direction, a continuity realiser provides a sequence $(x_{n})_{n\in \N}$ consisting of the points of discontinuity of $f\in BV\setminus AC$.  
Use Feferman's $\mu$ to find the least $n\in \N$ such that $f(x_{n})-f(0)> \int_{0}^{x_{n}}f'(x) ~dx $, if such $n$ exists.  The latter integral can be computed thanks to the (given) modulus of integrability.  
If there is no such $n\in \N$, we must have $f(y)-f(0)> \int_{0}^{y}f'(x) ~dx $ for some $y\in [0,1]$ where $f$ is continuous at $y$.   Since $\lambda z.\int_{0}^{z}f'(x) ~dx$ is continuous on $[0,1]$, there must be
a \emph{rational} $q\in [0,1]$ (namely close to the aforementioned $y\in [0,1]$) such that $f(q)-f(0)> \int_{0}^{q}f'(x) ~dx $.  Since we can find such a rational using $\mu^{2}$, we are done. 

\smallskip

For the reverse direction, proceed as in the proof of Theorem \ref{iop3}.  
Indeed, let $X\subset [0,1]$ be finite and consider $f:=\mathbb{1}_{X}$, which is in $BV$. 
We may assume $X\cap \Q=\emptyset$ as $\mu^{2}$ can enumerate all rationals in $X$.  
Clearly, $f$ has zero derivative a.e.\ and $\int_{a}^{b}f'(x) dx=0$ for any $a, b\in [0,1]$.  
For any $y\in (0,1]$ such that $f(y)-f(0)> \int_{0}^{y}f'(x) ~dx $, we must have $y\in X$.  
Now remove this real from $X$ and repeat the procedure until $X$ is empty. 
Note that in case $X=\emptyset$, the real $y\in [0,1]$ produced by the functional in the theorem satisfies $f(y)=0$, which is decidable given $\exists^{2}$.
\end{proof}
The previous theorem is not unique: we show in Sections \ref{purki2}-\ref{flukkkk} that other characterisations of $AC$ yield results similar to Theorem \ref{iop}.
The same holds for characterisations of Lipschitz continuous functions, as follows. 
\subsubsection{Lipschitz continuity}\label{hu3}
We study the computational properties of the second fundamental theorem of calculus restricted to Lipschitz continuous functions.

\smallskip

First of all, let $Lip$ be the class of \emph{Lipschitz continuous} functions, as follows.
\bdefi  
A function $f:[0,1]\di \R$ is \emph{Lipschitz continuous} on $[0,1]$ if there exists $c\in \R^{+}$ with $(\forall x, y\in [0,1])(|f(x)-f(y)|\leq c|x-y|)$.
\edefi
We have $Lip\subset AC$, but the former class can also be singled out as follows. 
\begin{thm}[\cite{voordedorst}*{Thm.\ 3.20} ]
A function is Lipschitz continuous on $[0,1]$ iff $f'\in L_{\infty}$ and for all $y\in [0,1]$, we have
\be\label{sim}\textstyle
f(y)-f(0)=\int_{0}^{y} f'(x) dx.
\ee
\end{thm}
As noted in the previous section, we always have `$\geq$' in \eqref{sim} for $BV$-functions.
These observations yields the following result, similar to Theorem \ref{iop}.
\begin{thm}\label{cordefakker}
The following is in the $\Omega$-cluster:
a functional that on input $f \in BV$ with effectively Riemann integrable $f' \in L_{\infty}$, outputs $y\in [0,1]$ which satisfies $f(y)-f(0)>\int_{0}^{y} f'(x) ~dx$ in case $f\not \in Lip$. 
\end{thm}
\begin{proof}
For the reverse direction, note that $\mathbb{1}_{X}$ for finite $X\subset [0,1]$ is not Lipschitz (unless $X=\emptyset$) while the $0$-function is clearly bounded (and hence in $L_{\infty}$).
For the foward direction, proceed as in the proof of Theorem \ref{iop}.
\end{proof}
We now sketch similar results based on the above; details may be found in \cite{voordedorst}.
\begin{rem}[Subspaces of $BV$]\rm
First of all, Riesz' notion of `$p$-variation' gives rise to the space $RBV_{p}$, strictly intermediate between $Lip$ and $AC$ for $1<p<\infty$.  
Riesz also showed that $RBV_{p}$ contains exactly those $AC$-functions with derivatives in $L_{p}$ (see \cite{voordedorst}*{Thm.\ 3.34}).  
Hence, we could repeat the above results for $p$-variation and $L_{p}$-integrability, and we would obtain the same results as in Theorem \ref{cordefakker}.
The same holds for the Riesz-Medvedev variation, which generalises Riesz' variation.  

\smallskip

Secondly, the H\"older continuous functions (called $\alpha$-Lipschitz in \cite{voordedorst}) form an intermediate space between $Lip$ and $AC$, i.e.\ there 
are results like Theorem \ref{cordefakker} involving H\"older continuity.

\smallskip

Thirdly, for $f\in AC$, we have $f\in Lip$ iff $|f'|$ is bounded (\cite{royorg}*{p.\ 112, \S20.b}).  Hence, we can reformulate the results in this section using the condition of having a bounded derivative (in the sense of Definition \ref{kood}).
By \cite{poki}*{Theorem 1.1}, there are \emph{many} other equivalent characterisations of $AC$.  
Another alternative characterisation of $Lip$ is via the notion of \emph{super bounded variation} (\cite{rijnwand}*{Thm 1.1.22}).   

\smallskip

Finally, in the next sections, we study (more) well-known characterisations of $AC$.  Since our results are similar to Theorem \ref{iop}, we shall be brief.
\end{rem}

\subsubsection{Connecting variation and integration}\label{purki2}
In this section, we study the computational properties of a characterisation of $AC$ that connects variation and integration.
\begin{thm}[\cite{voordedorst}*{p.\ 237}]\label{liko2}
A $BV$-function $f:[0,1]\di \R$ is in $AC$ if and only if 
\be\label{koit}\textstyle
V_{a}^{b}(f)=\int_{a}^{b}|f'(x)|dx, 
\ee
for all $0\leq a\leq b\leq 1$. We always have `$\geq$' in \eqref{koit} for $BV$-functions.
\end{thm}
In this light, we can formulate the computational task in the following theorem. 
\begin{thm}
The following is in the $\Omega$-cluster:
a functional that on input $f \in BV$ with effectively Riemann integrable $f'$, outputs $y\in [0,1]$ which satisfies $V_{0}^{y}(f)>\int_{0}^{y}|f'(x)| dx$ in case $f\not \in AC$. 
\end{thm}
\begin{proof}
For the reverse direction, the proof of Theorem \ref{iop} is readily adapted.  
For the forward direction, it is known that $f$ is continuous at $y\in [0,1]$ iff $\lambda z.V_{0}^{z}(f)$ is continuous at this point (\cite{voordedorst}*{Prop.\ 1.17, p.\ 60}).
Thus, we can proceed in the same way as in the (first part of the) proof of Theorem \ref{iop}.  Note that a continuity realiser functional allows us to compute $\lambda z.V_{0}^{z}(f)$ by Theorem \ref{thm.surprise}.
\end{proof}

\subsubsection{Absolute continuity and measure zero sets}
In this section, we study the computational properties of a characterisation of $AC$ involving the so-called Lusin $N$-property (see e.g.\ \cite{luser}), defined as follows.
\bdefi
A function $f: [a, b]\di \R$ has the \emph{Lusin $N$-property} if for all measure zero sets ${\displaystyle N\subset [a,b]}$, the set $f(N)=\{ f(x): x\in N \}$ is also measure zero. 
\edefi
\begin{thm}[Vitali-Banach-Zaretskij; \cite{voordedorst}*{Thm.\ 3.9}]
We have that $f\in AC$ is equivalent to the combination of:
\begin{itemize}
\item $f$ is continuous and BV on $[0,1]$,
\item $f$ satisfies the Lusin $N$-property on $[0,1]$.
\end{itemize}
\end{thm}
We could add the Lusin $N$-property to the comptutational task in Theorem \ref{iop}.  
We also have the following more interesting theorem. 
\begin{thm}
The following is in the $\Omega$-cluster:
a functional that for $f \in BV$ with the Lusin $N$-property, outputs $x\in [0,1]$ where $f$ is discontinuous if $f\not \in AC$. 
\end{thm}
\begin{proof}
The forward direction is immediate as a continuity realiser enumerates all points of discontinuity of $BV$-functions.  
For the reverse direction, $\mathbb{1}_{X}$ for finite $X\subset [0,1]$ has the Lusin $N$-property as its range is included in $\{0,1\}$. 
Now proceed as in the (second part of the) proof of Theorem \ref{iop3}.  
\end{proof}
\subsubsection{Absolute continuity and arc length}\label{flukkkkq}
In this section, we study the computational properties of a characterisation of $AC$ that involves the fundamental notion of `arc length' or `length of a curve'.
The (modern) notion of arc length as in \eqref{puhe} was already studied for discontinuous regulated functions in 1884, namely in \cite{scheeffer}*{\S1-2}, where it is also claimed to be essentially equivalent to Duhamel's 1866 approach from \cite{duhamel}*{Ch.\ VI}.  Around 1833, Dirksen, the PhD supervisor of Jacobi and Heine, already provides a definition of arc length that is (very) similar to \eqref{puhe} (see \cite{dirksen}*{\S2, p.\ 128}), but with some conceptual problems as discussed in \cite{coolitman}*{\S3}.

\smallskip

First of all, Jordan proves in \cite{jordel3}*{\S105} that $BV$-functions are exactly those for which the notion of `length of the graph of the function' makes sense.  In particular, $f\in BV$ if and only if the `length of the graph of $f$', defined as follows:
\be\label{puhe}\textstyle
L(f, [0,1]):=\sup_{0=t_{0}<t_{1}<\dots <t_{m}=1} \sum_{i=0}^{m-1} \sqrt{(t_{i}-t_{i+1})^{2}+(f(t_{i})-f(t_{i+1}))^{2}  }
\ee
exists and is finite by \cite{voordedorst}*{Thm.\ 3.28.(c)}.  In case the supremum in \eqref{puhe} exists (and is finite), $f$ is also called \emph{rectifiable}.  
We note that \cite{voordedorst}*{Thm.\ 3.28} contains another interesting property, namely that 
\be\label{flunk}
V_{a}^{b}(f)\leq L(f, [a, b])\leq V_{a}^{b}(f)+b-a, 
\ee
which implies that $ \lambda x.L(f, [a, x])$ is continuous at a point  $y\in (a,b)$ iff $\lambda x. V_{a}^{x}(f)$ is continuous at this point. 
As noted above, the latter is equivalent to $f$ being continuous at $y\in (a, b)$.  
Moreover, \eqref{flunk} suggests that computing $V_{a}^{b}(f)$ or computing $L(f, [a,b])$ amounts to the same, modulo say $\exists^{2}$.  

\smallskip

Secondly, we mention the following theorem that characterises $AC$ as those functions for which the arc length equals the well-known integral formula as in \eqref{flingo}.
\begin{thm}[Tonelli; \cite{voordedorst}*{p.\ 238}]
A $BV$-function $f:[0,1]\di \R$ is in $ AC$ iff 
\be\label{flingo}\textstyle
L(f, [0,1])=\int_{0}^{1} \sqrt{1+(f'(x))^{2}}dx.
\ee
We always have `$\geq$' for $BV$-functions in \eqref{flingo}.
\end{thm}
\noindent
We now have the following theorem. 
\begin{thm}
The following functionals are in the $\Omega$-cluster:
\begin{itemize}
\item a functional that for $f \in BV$ with effectively integrable $f'$, outputs $y\in [0,1]$ which satisfies $L(f, [0,y])>\int_{0}^{y} \sqrt{1+(f'(x))^{2}}dx$ if $f\not \in AC$,
\item a functional that for input $f \in BV$ and $y\in [0,1]$, outputs $L(f, [0,y])$,
\item a functional that for input $f \in BV$, outputs $n\in \N$ such that $L(f, [0,1])\leq n$.
\end{itemize}  
\end{thm}
\begin{proof}
In light of \eqref{flunk} and item \eqref{waard2} from Theorem \ref{thm.surprise}, the second and third item clearly belong to the $\Omega$-cluster.
For the remaining item, recall that $f$ is continuous at $y\in (0,1)$ iff $L(f, [0,y])$ is continuous at $y$, as noted below \eqref{flunk}.
Hence, we may proceed as in the proof of Theorem \ref{iop} to obtain the first item.  
To obtain $\Omega$ from the first item, the proof of Theorem \ref{iop} is readily modified, as in the above.
\end{proof}

\subsection{Functionals related to subspaces of $\BV$}\label{flukkkk}
In this section, we (briefly) study the computational properties of twp subclasses of $BV$, namely the \emph{Sobolev space} $W^{1, 1}$ (Section \ref{juik}) and the pseudo-monotone functions (Section \ref{juik2}).  
It is well-known that $AC\subsetneq W^{1, 1}\subsetneq BV$, with natural counterexamples, as also studied in second-order RM (see \cite{kreupel}).

\subsubsection{Functional related to Sobolev spaces}\label{juik}
First of all, as discussed in Section~\ref{KLM}, $BV$-functions are differentiable a.e., i.e.\ it makes sense to talk about `the derivative $f'$' of $f\in BV$, as $f'$ is unique up to a set of measure zero.
However, we may have $f'\not \in L_{1}$ for $f\in BV$, and $W^{1, 1}$ is essentially the subspace of $BV$-functions with \emph{integrable derivatives}, with a very specific technical meaning for the latter.  
In particular, $W^{1, 1}$ collects those $L_{1}$-functions with \emph{weak derivative} in $L_{1}$ (see \cite{ziemer}*{Def.\ 2.1.1}), though there is a connection to `classical' derivatives by \cite{ziemer}*{Thm.\ 2.1.4}.   Sobolev spaces play an important role in PDEs and have their origin in mathematical physics (\cite{sobolev1, sobolev2}).

\smallskip

Secondly, the following theorem implies that we may replace `$BV$' by `$W^{1,1}$' in Theorem \ref{thm.surprise} and still obtain computational equivalences.  
To be absolutely clear, our notion `$f\in W^{1, 1}$' refers to the usual\footnote{For $u\in L_{1}$, any $v \in L_{1}$ is a \emph{weak derivative} of $u$ if ${\textstyle \int _{0}^{1}u(t)\varphi '(t)\,dt=-\int _{0}^{1}v(t)\varphi (t)\,dt}$
for all infinitely differentiable functions $\varphi:[0,1]\di \R$ with $\varphi(0)=\varphi(1)=0$.} definition found in the literature.  
\begin{thm}\label{cordefakkers}
The following computational tasks are in the $\Omega$-cluster. 
\begin{itemize} 
\item For $f \in W^{1, 1}$, provide an enumeration of all points in $[0,1]$ where $f$ is discontinuous. 
\item For $f \in W^{1, 1}$, find $\sup_{x\in [0,1]}f(x)$. 
\end{itemize}  
\end{thm}
\begin{proof}
Since $W^{1, 1}\subset BV$, we only need to show that the items compute e.g.\ $\Omega_{\fin}$.  
Now, the function $f:=\mathbb{1}_{X}$ for finite $X\subset [0,1]$ is clearly in $W^{1, 1}$, since the weak derivative can be taken to be the zero everywhere function.  
In this way, the first item yields $\Omega_{\fin}$, and the same readily follows for the second item. 
\end{proof}
Finally, there are other classes related to $BV$ and $W^{1,1}$ which seem worth exploring, following Theorem \ref{cordefakkers}.
\begin{rem}\rm
First of all, $SBV$ is the class of \emph{special functions of bounded variation} with associated textbook references \cite{langebaard2} and \cite{ambrozijn}*{Ch.\ 4}.
This space is of importance in the study of the \emph{Mumford-Shah functional}, defined in the previous references.  
Since $SBV$ is intermediate between $W^{1, 1}$ and $BV$ (see \cite{ambrozijn}*{p.\ 212}), we immediately obtain Theorem \ref{cordefakkers} for $W^{1,1}$ replaced by $SBV$.  
In light of the rather technical definition of the latter, we only mention this result in passing. 

\smallskip

Secondly, so-called \emph{fractional} Sobolev spaces (see \cite{nezzahitch} for an overview) generalise Sobolev spaces to non-integer indices.  
For instance, we have $W^{1,1}\subsetneq W^{s, 1}$ for $0<s<1$ by \cite{nezzahitch}*{Prop.\ 2.2}.
We believe this field to have a lot of potential, but cannot offer more than that. 
\end{rem}
\subsubsection{Functionals related to pseudo-monotone functions}\label{juik2}
We show that we may weaken `$BV$' in the second cluster theorem to the subclass of \emph{pseudo-monotone functions}, which is intermediate between monotone and $BV$ and which was introduced in \cite{heiligejozef} under a different name.  
\begin{defi}[\cite{voordedorst}*{Def.\ 1.14}]\label{okra}
A function $f:[0,1 ]\di \R$ is \emph{pseudo-monotone} if it is bounded and if there is $n\in \N$ such that for any $[c, d]\subset \R$, the set $f^{-1}([c, d])$ is a union of $n$ \(closed, open, half-open, or singleton\) intervals.
\edefi
The fundamental result here is that $f\circ g$ is in $BV$ for all $f$ in $BV$ iff  $g$ is pseudo-monotone (\cite{heiligejozef}*{Theorem 3}).  The following basic $BV$-function 
\[
f(x):=
\begin{cases}
x^{2}\sin(\frac{1}{x}) & 0<x\leq 1\\
0 & x=0 
\end{cases},
\]
is not pseudo-monotone by \cite{voordedorst}*{Ex.\ 1.16}.  The following extension of the second cluster theorem implies that we may replace `$BV$' by `pseudo-monotone' therein.
\begin{thm}
The following computational tasks belong to the $\Omega$-cluster.
\begin{itemize}
\item For pseudo-monotone $f:[0,1]\di \R$, find $\sup_{x\in [0,1]}f(x)$.
\item For pseudo-monotone $f:[0,1]\di \R$, find $n\in \N$ such that $V_{0}^{1}(f)\leq n$.
\item For pseudo-monotone $f:[0,1]\di \R$, find $n\in \N$ such that for any $[c, d]\subset \R$, the set $f^{-1}([c, d])$ is a union of at most $n$ \(closed, open, half-open, or singleton\) intervals.
\end{itemize}
\end{thm}
\begin{proof}
That $\Omega$ can perform the first two itemised operations is immediate from the second cluster theorem (Theorem \ref{thm.surprise}).
For the third item, since pseudo-monotone functions are in $BV$, the first cluster theorem yields a sequence enumerating all points of discontinuity. 
With this is in place, we may (equivalently) replace the quantifier over the real numbers $c, d$ in Definition \ref{okra} by a quantifier over $\Q$.
Hence, $\mu^{2}$ can now find the number $n\in \N$ as in Definition \ref{okra}.

\smallskip

Now let $X\subset [0,1]$ be finite and note that $\mathbb{1}_{X}$ is pseudo-monotone.  The supremum of $\mathbb{1}_{X}$ lets us 
decide whether $X$ is empty (yielding $\Omega_{\textbf{b}}$) and an upper bound for $V_{0}^{1}(\mathbb{1}_{X})$ or for the number $n$ from Definition \ref{okra} provides an upper bound
on the number of elements in $X$ (yielding $\Omega_{\geq \#\fin}$).  The first cluster theorem and Lemma~\ref{borkim} thus yield $\Omega$ in each case, establishing membership of the $\Omega$-cluster. 
\end{proof}

\subsection{Functionals related to Caccioppoli sets}\label{caccio}
In this section, we study the computational properties of \emph{Caccioppoli sets}, aka \emph{finite perimeter} sets, which are essentially sets with characteristic function of bounded variation.    
This concept was pioneered by Caccioppoli himself in \cites{casino, casino2} while textbook references are \cite{maggi}, \cite{withmoregusto}*{Ch.\ 1}, and \cite{ambrozijn}*{\S3.3, p.\ 143}.
A historical sketch of the topic may be found in \cite{miranda}, including the seminal contributions of E.\ de Giorgio.  A novel characterisation of perimeter,  \emph{independent} of the theory of distributions, is presented in \cite{ambrozijn2}.

\smallskip

First of all, we need to discuss the exact definition of bounded variation in this context.  
Indeed, Caccioppoli sets are (always) defined in terms of the total variation \emph{of functions of several variables} as in e.g.\ \cite{ambrozijn}*{Def.\ 3.4}.  
The latter notion is (heavily) based on the theory of distributions, but there fortunately is a rather intimate 
connection to the so-called pointwise variation as in \eqref{tomb} for the one-dimensional case.  
In particular, an $[0,1]\di \R$-function that has bounded variation in the sense of item~\eqref{donp} of Definition \ref{varvar}, also has finite total variation in the generalised/distributional sense of \cite{ambrozijn}*{Def.\ 3.4}, which is immediate by \cite{ambrozijn}*{Thm~3.27}.  In this light, we can use our above definition of bounded variation as follows.  
\bdefi[Caccioppoli set]
A (measurable) set $E\subset \R$ is a \emph{classical Caccioppoli set} if its characteristic function $\mathbb{1}_{E}$ has bounded variation as in item \eqref{donp} of Definition \ref{varvar} on any $[a,b]\subset \R$.
\edefi
The word `classical' in the previous definition of course refers to the use of the `classical' definition of $BV$ due to Jordan.  
We will often omit this adjective.   

\smallskip

Secondly, regarding the place of Caccioppoli sets in the mathematical pantheon, we recall the following two facts.
\begin{itemize}
\item A subset of the reals is closed if and only if its characteristic function is upper semi-continuous.
\item A closed set in the sense of second-order RM (\cite{simpson2}*{II.5.6}) has a (code for a) continuous characteristic function (\cite{simpson2}*{II.7.1}).  
\end{itemize}
In this light, the study of closed Caccioppoli sets is rather close to the study of RM-closed sets, as $BV$-functions only have countably many points of discontinuity.  
Nonetheless, we have Theorem \ref{seal} where RM-closed sets are defined in \cite{simpson2}*{II.5.6}. 
\begin{thm}\label{seal}
The following functionals belong to the $\Omega$-cluster. 
\begin{enumerate}
 \renewcommand{\theenumi}{\roman{enumi}}
\item A functional that takes as input a \(closed\) Caccioppoli set $C\subset \R$, an upper bound on $C$, and returns $\sup C$.  \label{naartoegooien}
\item A functional that takes as input a closed Caccioppoli set $C\subset \R$ and returns an RM-code for $C$.  \label{klak}
\item A \(Tietze\) functional that takes as input a closed Caccioppoli set $C\subset [0,1]$ and $f:[0,1]\di \R$ continuous on $C$, and outputs a continuous $g:[0,1]\di \R$ such that $f=g$ on $C$ and $\sup_{C} f=\sup_{[0,1]} g$.  \label{flam}
\item A \(Urysohn\) functional that takes as input disjoint closed Caccioppoli sets $C_{i}\subset \R$ for $i=0,1$ and outputs continuous $f:\R\di \R$ with $f=i$ on $C_{i}$.  \label{flam4}
\item A \(Weierstrass\) functional that takes as input a closed Caccioppoli set $C\subset [0,1]$ and $f:[0,1]\di \R$ continuous on $C$, and outputs $x\in C$ such that $(\forall y\in C)(|f(y)|\leq |f(x)|)$.\label{flam2}
\end{enumerate}
\end{thm}
\begin{proof}
We first prove that a continuity realiser computes the items in the theorem.  
Given such a realiser and a set $C$ as in item \eqref{naartoegooien}, the supremum $\sup C$ equals a supremum over $\N$ (and $\Q$).  Indeed, the points of discontinuity of $\mathbb{1}_{C}$ are given 
by the continuity realiser, while the continuity points only have to be considered (for the supremum) if they are rational.  One similarly obtains an RM-code as in item~\eqref{klak} for $C$: for $x\in [0,1]\setminus C$, we can find $N\in \N$ such that $B(x, \frac{1}{2^{N}})\subset \big([0,1]\setminus C\big)$ as the quantifier over $\R$ in the latter formula may be replaced by a quantifier over $\N$ (and $\Q$).
For item \eqref{flam}, the various (second-order) versions of Tietze's extension theorem can all be proved in $\ACA_{0}$ (\cite{withgusto}) and $\exists^{2}$ readily computes an RM-code for the function $f$ restricted to the RM-code of $C$, as well as the extension function $g:[0,1]\di \R$ from the aforementioned codes.  The same approach works for item~\eqref{flam4} by \cite{simpson2}*{II.7.3}.  
Item \eqref{flam2} is proved in the same way as $\exists^{2}$ computes maxima of continuous functions by \cite{kohlenbach2}*{Prop.\ 3.14}.

\smallskip

For the reversals, fix a finite set $X\subset [0,1]$, which is closed and for which $\mathbb{1}_{X}$ has bounded variation. 
Using item \eqref{naartoegooien}, the real $\sup X$ yields an element of $X$, and repeating this procedure we may enumerate $X$, as required for $\Omega_{\fin}$, which is in the $\Omega$-cluster by Theorem \ref{thm.surprise}.  Similarly, $f(x):=\mathbb{1}_{X}(x)$ is continuous on $X$, and the continuous extension $g:[0,1]\di \R$ provided by item \eqref{flam} allows us to decide if $X$ is empty or not (by checking whether $(\exists q\in \Q\cap [0,1])(g(q)>0)$), i.e.\ we obtain $\Omega_{\textup{b}}$ and Lemma \ref{borkim} finishes the proof.  Items \eqref{flam4} and \eqref{flam2} yield $\Omega_{\textup{b}}$ in the same way.  
\end{proof}
We finish this section with a remark on related topics.  
\begin{rem}\rm
Regarding item \eqref{flam}, such functionals have been studied for about fifty years under the name \emph{simultaneous extenders}, starting -it seems- with \cites{lutzer}.  
The notion of Caccioppoli set is also central to the \emph{Pfeffer integral}, which is intermediate between the Lebesgue and gauge integral when restricted to the real line (see \cite{peper, peper2}).  
We have studied the logical properties of the gauge integral and associated Cousin's lemma/Heine-Borel theorem in \cite{dagsamIII}.
A related type of integral is defined in \cite{capo, houhethurz} based on \emph{partitions of unity} involving $BV$-functions.  Such partitions (involving continuous functions) have been studied in both second- and higher-order RM (see \cite{simpson2}*{p.\ 89} and \cite{sahotop}).
\end{rem}

\begin{ack}\rm
Our research was supported by the \emph{Deutsche Forschungsgemeinschaft} via the DFG grant SA3418/1-1.
Initial results were obtained during the stimulating MFO workshop (ID 2046) on proof theory and constructive mathematics in Oberwolfach in early Nov.\ 2020.  
We express our gratitude towards the aforementioned institutions.    
Finally, we thank John Longley for spotting the error in our lambda calculus and his advise regarding the correction published in \cite{dagsamXV}.
\end{ack}

\bye